\documentclass[sigplan,9pt]{acmart}

\acmConference[Technical report]{}{Ar{X}iv}{2018}
\acmYear{}
\acmISBN{} 
\acmDOI{} 
\startPage{1}

\setcopyright{none}

\usepackage{booktabs}   
\usepackage{subcaption} 

\usepackage{tikz}
\usetikzlibrary{arrows}
\usepackage{pict2e}

\usepackage{amsmath}
\usepackage{amsthm}
\usepackage{amssymb}
\usepackage{latexsym}
\usepackage{proof}
\usepackage{microtype}

\usepackage{bussproofs}
\usepackage{amssymb}
\usepackage{amsmath}
\usepackage{mathrsfs}
\usepackage{upgreek}
 \usepackage{graphicx}

\usepackage{color}

\newdimen\CdotAxis
\newcommand*{\CdotAux}[3]{%
  {%
    \settoheight\CdotAxis{$#2\vcenter{}$}%
    \sbox0{%
      \raisebox\CdotAxis{%
        \scalebox{#1}{%
          \raisebox{-1.1pt}{%
            $\mathsurround=0pt #2#3$%
          }%
        }%
      }%
    }%
    \dp0=0pt %
    \sbox2{$#2\bullet$}%
    \ifdim\ht2<\ht0 %
      \ht0=\ht2 %
    \fi
    \sbox2{$\mathsurround=0pt #2#3$}%
    \hbox to \wd2{\hss\usebox{0}\hss}%
  }%
}

\newcommand{\IL}                      {{\mathrm{IL}}}
\newcommand{\LC}     {{\mathsf{G}}}

\newcommand{\proj}                     { {\mathsf{\uppi}} }

\newcommand{\pair}[2]{\langle #1,#2\rangle}

\newcommand{\nf}{\mathsf{NF}}

\newcommand{\Ecrom}[3]{#2 \parallel_{#1} #3}

\newcommand{\sq}[1]  {{\boldsymbol{#1}}}

\newcommand{\ehi}{\color{red}}

\newcommand{\IMPL}{\rightarrow}
\newcommand{\ET}{\wedge}
\newcommand{\VEL}{\vee}

\newcommand{\impl}{\rightarrow}
\newcommand{\et}{\wedge}
\newcommand{\vel}{\vee}
\newcommand{\fal}{\bot}
\newcommand{\ver}{\top}
\newcommand{\non}{\neg}

\newcommand{\ax}{\mathbf{Ax}}
\newcommand{\exm}{\mathrm{EM}}

\newcommand{\lamg}{\lam _{\mathrm{G}}}
\newcommand{\lama}{{\lambda _{\ax} }}
\newcommand{\lamem}{{\lambda _{\exm} }}

\newcommand{\inappendix}[1]{}

\newcommand{\lamBW}{\lam_{\mathrm{C_k}}} 
\newcommand{\lamBR}{\lam _{\mathrm{\exm_n}}}

\newcommand{\lam}{\lambda}
\newcommand{\lan}{\langle}
\newcommand{\ran}{\rangle}
\newcommand{\efq}[2]{#2 \, \mathsf{efq}_{#1 }}
\newcommand{\exfalso}{\mathsf{efq}}
\newcommand{\p}{\parallel}
\newcommand{\chosen}[1]{\underline{ #1 }}
\newcommand{\send}[1]{\overline{#1}\,}
\newcommand{\inj}{\upiota}
\newcommand{\termt}{\mathrm{tt}}
\newcommand{\pp}[2] {  {}_{#1} \hspace{-1pt} ( #2 ) }
\newcommand{\mapstopar}{\rightrightarrows}

\newcommand{\NJ}{\mathbf{NJ}}

\newcommand{\NAx}{\mathbf{NAx}}

\theoremstyle{plain}
\newtheorem{theorem}{Theorem}[section]
\newtheorem*{theorem*}{Theorem}
\newtheorem{proposition}[theorem]{Proposition}

\theoremstyle{definition}
\newtheorem{definition}{Definition}[section]
\newtheorem{example}{Example}[section]
\theoremstyle{remark}
\newtheorem{remark}{Remark}[section]

\usepackage{xcolor}

\usepackage{tikz}
\usetikzlibrary{arrows}
\usetikzlibrary{decorations.pathmorphing}

\def\NJ{{\bf NJ}}

\newcommand{\seq}{\Rightarrow}

\def\hh{\ |\ } 
\def\Deduce#1{\hbox{$\hphantom{#1}$\kern\inferLabelSkip\DeduceSym\kern\inferLab
elSkip$#1$}}

\newcommand{\logic}[1]{\ensuremath{\mathbf #1}} 
\def\LC{\logic{G}}

\begin{document}

\title[Disjunctive Axioms and Concurrent $\lambda$-Calculi: a Curry-Howard Approach]{Disjunctive Axioms and Concurrent $\lambda$-Calculi:\\ a Curry-Howard Approach}
\titlenote{Supported by FWF: grants Y544-N2 and W1255-N23.}             


\author{Federico Aschieri}
\affiliation{              
\institution{TU Wien, Vienna (Austria)}            
}

\author{Agata Ciabattoni}
\affiliation{              
\institution{TU Wien, Vienna (Austria)}            
}

\author{Francesco A.\ Genco}
\affiliation{              
\institution{TU Wien, Vienna (Austria)}            
}
\begin{abstract}
We add to intuitionistic logic 
infinitely many 
classical disjunctive
tautologies and use the Curry--Howard correspondence to obtain typed
concurrent $\lambda$-calculi; each of them features a specific
communication mechanism, including broadcasting and cyclic message-exchange, 
and enhanced expressive power w.r.t.\  the $\lambda$-calculus.
Moreover they all implement forms of code mobility.
Our results provide a first concurrent computational interpretation
for many  
propositional intermediate logics, classical logic included.
\end{abstract}

\maketitle

\section{Introduction}
\label{Introduction}

Although definable in three lines, the $\lambda$-calculus provides an elegant, yet powerful 
computational theory.
Formally introduced by Church, it has been discovered three times,
giving rise to a beautiful connection between three different fields.
 In mathematics, it was informally used as a simple syntax 
for function definition and application. In computer science, it was introduced as a functional
programming language. In logic, it unveiled the computational side of natural deduction: simply
typed $\lambda$-terms are isomorphic to natural deduction proofs and evaluation of programs corresponds to normalization of 
such proofs. This latter connection is known as Curry--Howard correspondence~\cite{Howard}.
%
%
\vspace{-0.1cm}
\subsection*{Extending the $\lambda$-calculus}
The only means of computation in the $\lambda$-calculus is the $\beta$-reduction
rule which models function application.
This suffices for Turing-completeness, but renders the system sometimes inefficient and not expressive
enough. For instance, Scott's parallel OR cannot be implemented. 
In response, many extensions of the $\lambda$-calculus have been introduced. 
A well-know extension with control operators is the $\lambda \mu$-calculus~\cite{Parigot}. It arises by interpreting via the Curry--Howard correspondence
Pierce's axiom $((A \to B) \to A) \to A$ of classical logic, first used in 
\cite{Griffin} to provide a type for the call/cc operator of Scheme and in
the $\lambda$C-calculus~\cite{Felleisen}.
Particularly challenging is to extend the $\lambda$-calculus 
with mechanisms for message exchanges; the latter are at
the core of the $\pi$-calculus~\cite{Milner,sangiorgiwalker2003}
-- the most widespread formalism for modeling concurrent systems.
Concurrent extensions of the $\lambda$-calculus have been defined, 
for instance, 
by adding  communication mechanisms inspired by process calculi
(e.g.~\cite{Boudol89}) or by programming languages (e.g.~\cite{NSS2006}).
Logic--based extensions are mainly rooted in 
variants and extensions of linear logic.
Indeed the functional languages in~\cite{Wadler2012,TCP2013} use 
the Curry-Howard interpretation of linear logic as
session-typed processes to extend the $\lambda$-calculus (linear,
for~\cite{Wadler2012}) with primitives for $\pi$-calculus-based
session-typed communication~\cite{CairesPfenning}. 
Enhancing the base logics with suitable axioms allow
additional features and programming language constructs 
to be captured, see, e.g.,~\cite{CairesPerez}. 
A challenge is to find such axioms.
%
{\em One} of them was investigated in \cite{ACGlics2017}, where
a concurrent reading
of $(G)=(A \to B) \vee (B \to A)$ was proposed.
There, by lifting the original Curry-Howard correspondence in \cite{Howard}  to
G\"odel logic (intuitionistic logic IL extended with $(G)$) a new typed concurrent $\lambda$-calculus
was introduced.
$\lamg$, as it is called, turns out to be an extension of the simply typed $\lambda$-calculus
with a parallel operator connecting two processes and a mechanism for symmetric message-exchange. This mechanism is directly
extracted from the normalization proof for the natural deduction system for G\"odel logic, similarly to how the $\beta$-reduction arises from normalization of intuitionistic derivations.
As a result each construct in $\lamg$ has a counterpart in the logic, making programs isomorphic to logical proofs,
as opposed to the calculi in \cite{Wadler2012,TCP2013}. 
\vspace{-0,1cm}
\subsection*{Disjunctive Tautologies, Hypersequents and Concurrency}

In this paper we identify infinitely many axioms, each leading
to a different communication mechanism and use the Curry--Howard correspondence
to obtain new typed concurrent $\lambda$-calculi.
Our axioms belong to the classical tautologies 
interpreted in \cite{DanosKrivine} as synchronization schemes --
 (G) being {\em one} of them.
For these formulas, \cite{DanosKrivine} provides
realizers in  a concurrent extension of the $\lambda$C-calculus~\cite{Felleisen}, but 
the question of developing concurrent calculi based on them 
and Curry--Howard isomorphisms for the corresponding intermediate\footnote{These are logics
intermediate between intuitionistic and classical logic.} logics was left open.
A further and independent hint on the connection between these formulas and concurrency
comes from proof theory: they 
all belong to the class ${\mathcal P}_{2}$ of the classification in~\cite{CGT08}
and can therefore be
transformed into equivalent structural rules in the {\em hypersequent calculus}.
Avron -- who introduced this framework -- suggested 
in~\cite{Avron91} that IL extended with axioms equivalent to such rules
{\em could serve as bases for concurrent $\lambda$-calculi}.
Hypersequents are indeed nothing but parallel sequents whose interaction is governed by special rules
that allow the sequents to ``communicate'', e.g.,~\cite{Avron96}.
Different rules enable different types of communications between sequents.

Despite their built in ``concurrency'' and various attempts (e.g.~\cite{BP2015,Hirai}),
hypersequent rules do not seem to be good bases for a Curry--Howard interpretation, 
and Avron's claim has remained unsolved for more than 25 years.
What we need are instead 
natural deduction inferences that, if well designed, can be
interpreted as program instructions, and in particular as typed
$\lambda$-terms. Normalization~\cite{Prawitz}, which corresponds to
the execution of the resulting programs, can then be used to obtain
analytic proofs, i.e. proofs containing only formulas that are subformulas
of hypotheses or conclusion.
Attempts to define such natural deduction inferences capturing
some of the considered axioms are, e.g., in~\cite{L1982,BP2015};
nevertheless, in these texts normalization is missing or the 
introduced rules ``mirror'' the hypersequent structure, thus hindering
their computational reading. 
%
%
For the {\em particular case} of G\"odel logic  
a well-designed calculus was introduced in~\cite{ACGlics2017} by extending the natural deduction system $\NJ$ for IL 
with an inference rule simulating the hypersequent rule for $(G)$.
The reduction rules required to obtain analytic proofs
led to the definition of a {\em specific} typed concurrent
functional language.

%
\vspace{-0,1cm}
\subsection*{Curry--Howard correspondence: the concurrent calculi $\lama$}
In this paper we exploit the intuitions in \cite{Avron91,DanosKrivine} to define
infinitely-many logic-based typed concurrent $\lambda$-calculi.
The calculi are extracted -- using the Curry--Howard correspondence -- from 
$\IL$ extended with  classical tautologies  of the form
$\bigvee ( A_i \to B_i)  $
where $A_i$ is a conjunction of propositional
variables (possibly $\ver$)  and  $B_i$  is a propositional variable or $\fal$, 
and the conjuncts in all $A_i$'s are distinct. 
Examples of such formulas include $(\exm)$, $A \vee \neg A$, $(C_n)$ $(A_1 \to A_2) \vee \dots \vee (A_n \to A_1)$,
and $(G_n)$ $(A_1 \to A_2) \vee \dots \vee (A_{n-1} \to A_n) \vee \neg A_n$, for $n \geq 2$;
their addition to $\IL$ leads to classical logic, cyclic $\IL$~\cite{L1982} (G\"odel logic, for $n=2$), and $n$-valued G\"odel logic, respectively.
Let $Ax$ be any finite set of such axioms. 
The corresponding calculus $\lama$ arises from $\IL + \ax$ as the simply typed 
$\lambda$-calculus arises from $\IL$.
Examples of our $\lambda$-calculi $\lama$ are: 
\begin{description}
\item[$\lamem$:] the simplest message-passing mechanism \`a la
$\pi$-calculus 
\item[$\lamBW$:] cyclic communication among $k$ processes.
\end{description}
Notice that $\lamem$ provides a new computational interpretation of classical logic, obtained 
using $A \vee \neg A$ in the Curry--Howard correspondence. It was indeed a long-standing open problem to provide an analytic natural deduction based on $(EM)$ and enjoying a significant computation interpretation: see \cite{deGrooteex} for an attempt.

Furthermore, different syntactic forms of the same axiom lead to different $\lambda$-calculi, e.g.
for $(\exm)_{n} \; A \vee \dots \vee A \vee \neg A$ we have
\begin{description}
\item[$\lamBR$:] broadcast messages to $n$ processes.
\end{description}
The natural deduction calculus used as type system for
$\lama$ is
defined 
by extending $\NJ$ with the natural deduction version~\cite{CG2018} of the hypersequent rules 
for $\ax$.
The decorated  version of these rules lead to $\lambda$-calculi with parallel operators
that bind two or more processes, according to the shape of the rules.
For example the decorated version of the rules for $(EM)$ and $(C_3)$ are:
\[\small \infer{s \parallel _a t : F}{\infer*{t :
      F}{\infer{a^{\non A} \, u : \fal }{ u : A }}&\infer*{s : F}{[a^{A} : A]}}
\;\;\;
      \infer{ \pp{a}{ \, r \p s \p t } : F }{
       \infer*{r:F}{\infer{a^{A\impl B} \, u:B}{u:A}} &
      \infer*{s:F}{\infer{a^{B \impl C} \, v: C}{v:B}} &
      \infer*{t:F}{\infer{a^{C \impl A} \, w:A}{w:C}}}\]
The variable $a$ represents a private communication channel that behaves similarly
to the $\pi$-calculus restriction
operator $\nu$~\cite{Milner, sangiorgiwalker2003}. Our natural deduction calculi are defined in a modular way and 
the rules added to $\NJ$ are higher-level~\cite{schroederh2014}, as they also discharge rule
applications rather than only formulae. 
We provide a normalization proof that uniformly applies to all the associated $\lambda$-calculi. 
Our reduction rules and normalization strategy mark a radical departure from~\cite{ACGlics2017}.  
They are driven not only by logic, but also by programming.
Indeed, while in \cite{ACGlics2017}  processes can only communicate
when the type of their channels violates the subformula property,  our
new reductions are logic independent and are activated when messages
are values (cf.\ Def.~\ref{def:value}). 
As a result, programs are easier to write and and their evaluation exhausts all \emph{active sessions}.
The normalization proof employs sophisticated termination arguments and
works for the full logic language, $\vee$ included. Since it does not exploit properties specific to G\"odel logic,
it can be used in a general setting. 
%
%
Its instantiation leads to computational reductions that specify the mechanism for message-exchange in the
corresponding typed 
$\lambda$-calculus. 

For any logic $\IL + \textit{Ax}$ and $\lama$ we show the perfect match between computation steps and proof reductions
(Subject Reduction),  a terminating reduction strategy for $\lama$ (Normalization) and 
the Subformula Property.
Our logic-based calculi are more expressive than the $\lambda$-calculus, and
their typed versions are non-deterministic\footnote{To enforce
non-determinism in typing systems for the $\pi$-calculus, {\em extensions} of linear logic
are needed \cite{CairesPerez}.}.
Moreover they can implement multiparty communication (in contrast with $\lamg$~\cite{ACGlics2017}, cf. Sec.~\ref{sec:broad})
and, although deadlock-free, communication between cyclic interconnected processes  
(in contrast with~\cite{CairesPfenning}, cf.\ Ex.~\ref{Milnerscheduler}).
Furthermore the computational reductions associated to {\em each} axiom enjoy a natural interpretation in terms of
higher-order process passing and
 implement forms of code mobility~\cite{codemobil}; the latter can be used to improve 
efficiency of programs: open processes can be sent first and
new communication channels taking care of their closures can be created afterwards (Ex.~\ref{ex:optimize}). 
 

The paper is organized as follows: Sec.~\ref{sec:nd} defines the class of considered logics and
their natural deduction systems. In Sec.~\ref{sec:cl} we use the
typed concurrent calculus $\lamem$ for classical logic as a case study to explain the ideas behind the costruction
of our calculi, their reductions and the normalization proof. Examples of programs, including
broadcasting, are also given. Sec.~\ref{sec:ax} deals with the general case: the definition of
the typed concurrent calculi $\lama$ parametrized on the set of axioms
$\ax$. A sophisticated 
normalization proof working for all
of them is presented in Sec.~\ref{sec:norm}.



\section{Disjunctive Axioms and Natural Deduction}
\label{sec:nd}

We introduce natural deductions for infinitely many 
intermediate logics obtained by extending
$\IL$ by any set of Hilbert axioms 
\begin{equation}
\label{ouraxiom}
\bigvee_{i=1, \dots, k} (\bigwedge_{j=1, \dots, n_k} A_{i}^{j}) \to B_i 
\end{equation}where $A^j_i$ and $B_i$, for $i=1, \dots, k$ is a
propositional variable or $\top$ or $\bot$, with the additional
restriction\footnote{This
restriction enables a simpler notion of analiticity without ruling out
interesting computational mechanisms.}  that all $A^i_j$'s are
distinct and each $A^{i}_{j}\neq \top$ is equal to some $B_{l}$.
Modular analytic calculi for them are
algorithmically
\cite{CGT08} defined (https://www.logic.at/tinc/webaxiomcalc/) in 
the hypersequent framework, whose basic objects are disjunctions of
sequents which ``work in parallel'', see, e.g.,~\cite{Avron91,Avron96}.

Neither Hilbert nor hypersequent calculi are good bases for a Curry--Howard correspondence, 
for different reasons; Hilbert calculi are not analytic, and whereas
the hypersequent structure
enables analytic proofs it hinders a computational reading of its rules. 
Our natural deduction calculi are defined by suitably
reformulating the corresponding hypersequent calculi, which we describe below.
A uniform normalization proof, which works for all of them, 
will be carried out in Section~~\ref{sec:norm} in the $\lambda$-calculus setting.  
\begin{definition}
\textbf{Hypersequents} are multiset of sequents, written as
$  \Gamma_1 \seq \Pi_1 \hh \dots \hh \Gamma_n \seq \Pi_n$
where, for all $i = 1, \dots n,$ $\Gamma_i \seq \Pi_i$ is an ordinary
sequent, called component.
\end{definition}
The symbol ``$|$'' is a meta-level disjunction; this is reflected by the
presence in hypersequent calculi of the external structural rules $(em)$ and $(ec)$ of
weakening and contraction, operating on whole sequents, rather than on formulas.

Let $\ax$ be any finite subset of axioms of the form above. 
A cut-free hypersequent calculus for 
$IL + \ax$ is defined by 
adding structural rules equivalent to the provability of the axioms in $\ax$
to the base calculus, i.e. the hypersequent version of the {\em LJ} sequent calculus for $\IL$ with (ew) and (ec).
These additional rules allow the ``exchange of information'' between different components. Below we present the $(ec)$ rule and 
the rule $(cl)$ equivalent to $(\exm) \; A \vee \neg A$: 
$$\infer[(ec)]{G \hh \Gamma \seq \Pi}
{G \hh \Gamma \seq \Pi \hh \Gamma \seq \Pi}
\quad 
\quad 
\infer[(cl)]{G \hh \Gamma \seq \hh  \Gamma' \seq \Pi}
{G \hh \Gamma, \Gamma' \seq \Pi}
$$
In general hypersequent rules equivalent to axiom (\ref{ouraxiom}) 
are~\cite{CGT08}
$$\infer{G\hh \Gamma^1_1, \dots , \Gamma^1_{n_1}, \Gamma_1\seq \Pi_1 \hh \dots \hh 
\Gamma^1_k, \dots , \Gamma^k_{n_k}, \Gamma_k \seq \Pi_k}{G \hh
\Gamma_i, \Gamma^j_{p} \seq \Pi_i}$$
for $i, j=1, \dots,k$, $p \in \{1, \dots n_j\}$ and possibly 
$\{\Gamma_i, \Pi_i \} = \emptyset$.
Their reformulation in \cite{CG2018} as \emph{higher-level rules} 
\`{a} la~\cite{schroederh2014} (see rule~\ref{ourN} in Section~\ref{sec:ax})  
is the key for the definition of the natural deduction systems
$\NAx$. 
Some of these rules were considered in~\cite{CG2018}, \cite{deGrooteex,L1982}
but without any normalization proof, which was provided in~\cite{ACGlics2017} for 
a {\em particular instance}: the rule for 
$(A \to B) \vee (B \to A)$. Our proof for $\NAx$
is much more general and will treat G\"odel logic as a very particular case, also delivering a more expressive set of reductions.

\section{Case study: the $\lamem$-calculus}
\label{sec:cl}
We describe our typed concurrent $\lambda$-calculus $\lamem$  for Classical
Logic. Actually, a calculus for $\IL$ with axiom
$\non A \lor \top\IMPL A$ can  be automatically extracted from the
family of calculi $\lama$ that we introduce in Section~\ref{sec:ax}. However,
the corresponding reduction rules can be significantly simplified and
we present here directly the result of the simplification. Since
Normalization and Subject Reduction are already proved in Section
\ref{sec:norm} for the whole family  by a very general argument, we
shall not bother the reader with the simplified proof: we just remark
that the adaption is straightforward.

$\lamem$ extends the standard Curry--Howard correspondence~\cite{Wadler}
for $\NJ$ with a parallel operator that interprets the natural deduction
rule for $A \vee \neg A$. 
We describe $\lamem$-terms and their computational behavior and show examples of programs.

The table below defines a type assignment for $\lamem$-terms, called \textbf{proof
terms} and denoted by $t, u, v \dots$, which is isomorphic to $\NJ + \exm$ (see Sec.~\ref{Introduction}).
The typing rules for axioms, implication, conjunction, disjunction and ex-falso-quodlibet are  those
for the simply typed $\lambda$-calculus, while parallelism is introduced by the rule for axiom $\exm$.
The contraction rule is useful, although redundant, and is  
the analogous of the  external contraction rule
$(ec)$ of  hypersequent calculi, see Sec.~\ref{sec:nd}.
\begin{table}[h]
\hrule 
  \centering
  \begin{footnotesize}
  \begin{center}
    $\begin{array}{c} x^A: A
     \end{array} \text{ for $x$ intuitionistic variable}\quad $
$\qquad \vcenter{\infer[\text{contraction}]{ t_1 \p t_2 : A}{ t_1 : A
& t_2 : A }}$

\medskip

     $\vcenter{\infer{ \langle u,t\rangle: A \wedge B}{u:A & t:B}}
     \qquad\qquad \vcenter{\infer{u\,\pi_0: A}{u: A\wedge B}} \qquad \qquad
     \vcenter{\infer{u\,\pi_1: B}{u: A\wedge B}}$

     $ \vcenter{\infer{\lambda x^{A} u: A\rightarrow
         B}{\infer*{u:B}{[x^{A}: A]}}} \quad \vcenter{\infer{tu:B}{ t:
         A\IMPL B & u:A}} \quad $
     $\vcenter{\infer{ \efq{P}{u}: P}{ u: \bot}}\quad \text{with $P$
       atomic, $P \neq \bot$.}$

     $\infer{\inj_{0}(u): A\vee B }{u: A} \quad $
     $\infer{\inj_{1}(u): A\vee B}{u: B}\quad $
     $\infer{u\, [x^{A}.w_{1}, y^{B}.w_{2}]: C}{ u: A\lor B &
       \infer*{w_{1}: C}{[x^{A}: A]} & \infer*{w_{2}: C}{[y^{B}: B]}}$
    
     \medskip
      
$ a^A :A
\qquad $$\vcenter{\infer{a^{\non A} t :
    \fal}{t:A}}\qquad $$\vcenter{\infer[(\exm )]{u \parallel_{a} v:
B}{\infer*{u:B}{[a^{\scriptscriptstyle \non A}: \non
A]}&\infer*{v:B}{[a^{A}: A]}}}  $

\end{center}
where all the occurrences of $a$ in $u$ and $v$ are respectively of the form $a^{\non A}$ and $a^{A}$.
   \end{footnotesize}
\smallskip
\hrule 
\medskip
\caption{Type assignments for $\lamem$ terms.}\label{tab:type_em}
\vspace{-25pt}
\end{table}
%
%
%
%
%
%
%
We fix notation, definitions and terminology that will be used
throughout the paper. 
 
Proof terms may contain \emph{intuitionistic} variables $x_{0}^{A}, x_{1}^{A}, x_{2}^{A},
\ldots$ of type $A$ for every formula $A$; these variables are 
denoted as $x^{A},$ $ y^{A}, $ $ z^{A} , \dots$. Proof terms also contain \emph{channel} variables $a_{0}^{A}, a_{1}^{A}, a_{2}^{A},
\ldots$ of type $A$ for every formula $A$ ; these variables will be denoted as $ a^{A}, b^{A}, c^{A},\dots$ and represent a \emph{private} communication channel between the parallel processes. Whenever the type is not relevant, it will be dropped and we shall simply write $x, y, z, \ldots, a,
b, \dots$. Notice that, according to the typing rules, channels variables cannot occur alone and thus are not  typed terms, unlike intuitionistic variables. As convention variables $a^{\non A}$ will be denote as $\overline{a}$. 

The free and bound variables of a proof term are defined as usual and
for the new term $\Ecrom{a}{u}{v}$, all the free occurrences of $a$ in
$u$ and $v$ are bound in $\Ecrom{a}{u}{v}$. In the following we
assume the standard renaming rules and alpha equivalences that are
used to avoid capture of variables in the reduction rules.

\textbf{Notation}. The connective $\rightarrow$  and $\et $ associate to the right and 
by $\langle t_{1}, t_{2}, \ldots, t_{n}\rangle $ we
denote the term
 $\langle t_{1}, \langle t_{2}, \ldots \langle t_{n-1},
t_{n}\rangle\ldots \rangle\rangle$ (which is $\termt : \ver $ if $n=0$)
and by $\proj_{i}$, for $i=0,
\ldots, n$, the sequence $\pi_{1}\ldots \pi_{1}
\pi_{0}$ selecting the $(i+1)$th element of the sequence. 
The expression $A_{1} \ET
\dots \ET A_{n}$ denotes
$\top$ if $n=0$ and thus it is the empty conjunction.

Often, if $\Gamma= x_{1}: A_{1}, \ldots, x_{n}: A_{n}$ and all free
variables of a proof term $t: A$ are in $x_{1}, \ldots, x_{n}$, we
 write $\Gamma\vdash t: A$. From the logical point of view, $t$
represents a natural deduction of $A$ from the hypotheses $A_{1},
\ldots, A_{n}$. We shall write $\text{EM}\vdash t: A$ whenever $\vdash
t: A$, and the notation means provability of $A$ in propositional
classical logic. If the symbol $\parallel$ does not occur in it, then
$t$ is a \textbf{simply typed $\lambda$-term} representing an
intuitionistic deduction.


\begin{definition}[Simple Contexts]\label{defi-simplec}
A \textbf{simple context} $\mathcal{C}[\ ]$ is a simply typed
$\lambda$-term with some fixed variable $[]$ occurring exactly
once. For any proof term $u$ of the same type of $[]$,
$\mathcal{C}[u]$ denotes the term obtained replacing $[]$ with $u$ in
$\mathcal{C}[\ ]$, \emph{without renaming of bound variables}.
\end{definition}

 \begin{figure*}[!htb]
\footnotesize{
\hrule
\smallskip
\begin{flushleft}
  \textbf{Intuitionistic Reductions} \hspace{20pt}
  $(\lambda x^{\scriptscriptstyle A}\, u)t\mapsto
  u[t/x^{\scriptscriptstyle A}] \qquad
  \pair{u_0}{u_1}\,\pi_{i}\mapsto u_i, \mbox{ for $i=0,1$} \qquad 
  $ $ \inj _i (t) [x_0.u_0,x_1.u_1] \mapsto u_i [t/x_i]$
\end{flushleft}

\smallskip
 
\begin{flushleft}
  \textbf{Disjunction Permutations} \hspace{70pt} $t [x_0.u_0,x_1.u_1]
\xi \mapsto t [x_0.u_0\xi,x_1.u_1\xi] \qquad \text{ if }\xi\text{ is a
one-element stack}$
  \end{flushleft}

\begin{flushleft}
  \textbf{Parallel Operator Permutations} \hspace{10pt} $w(\Ecrom{a}{u}{v}) \mapsto
\Ecrom{a}{wu}{wv} \qquad $ $(\Ecrom{a}{u}{v}) \xi \mapsto \Ecrom{a}{u
  \xi}{v\xi} \qquad \text{ if $\xi$ is a one-element stack and $a$
  does not occur free in $w$}$
\end{flushleft} 



$\lambda x^{\scriptscriptstyle A}\,(\Ecrom{a}{u}{v}) \mapsto
\Ecrom{a}{\lambda x^{\scriptscriptstyle A}\,u}{\lambda
x^{\scriptscriptstyle A}\, v} \quad $ $ \inj _i ( u \p _a v ) \mapsto
\inj _i (u) \p _a \inj _i(v)\quad $ $\langle u \parallel_{a} v,\,
w\rangle \mapsto \langle u, w\rangle \parallel_{a} \langle v, w\rangle
\quad $ $\langle w, \,u \parallel_{a} v\rangle \mapsto \langle w,
u\rangle \parallel_{a} \langle w, v\rangle \; \text{ if $a$ does not
occur free in $w$}$

$(u\parallel_{a} v)\parallel_{b} w \mapsto (u\parallel_{b}
w)\parallel_{a} (v\parallel_{b} w) \qquad$
$w \parallel_{b} (u\parallel_{a} v) \mapsto (w\parallel_{b}
u)\parallel_{a} (w\parallel_{b} v) \quad \mbox{ if $u, v, w$ do not contain active sessions } $

\medskip 

  \textbf{Communication Reductions}

\begin{flushleft} \textbf{Activation Reduction} \hspace{100pt} $
u \parallel_{a} v \; \mapsto \; u [b/a] \parallel_{b} v [b/a]$

where $a$ is not active, $b$ is a fresh \emph{active} variable, and
there is some occurrence of $a$ in $u$ of the form $a
w$ for a value $w$.
\end{flushleft}

\smallskip 

\begin{flushleft}
  \textbf{Basic Cross Reductions} \hspace{130pt}
  $\mathcal{C}[\send{a}\, u]\parallel_{a} \mathcal{D} \ \mapsto \
  \mathcal{D}[ u /a ]   $ 

where $\send{a} : \non A, a : A $,  $\mathcal{C}[a\, u], \mathcal{D}$ are simply typed
$\lambda$-terms and $\mathcal{C}, \mathcal{D}$ simple contexts; the
sequence of free variables of $u$ is empty;
the displayed occurrence of $\send{a}$ is rightmost; $b$ is fresh
\end{flushleft}
\smallskip

\begin{flushleft}
  \textbf{Cross Reductions} \hspace{70pt}
  $u\parallel_{a}v \mapsto u, \mbox{ if $a$ does not occur in $u$}
  \quad \mbox{and} \quad u\parallel_{a}v \mapsto v, \mbox{ if $a$ does
    not occur in $v$}$
\end{flushleft}

  \smallskip 

$\mathcal{C}[\send{a}\, u]\parallel_{a} \mathcal{D} \ \mapsto \
    ( \mathcal{C}[\send{b} \, \lan \sq{y} \ran ] \parallel_{a} \mathcal{D} )
    \, \parallel_{b}\, \mathcal{D}[ u^{b / \sq{y}  } /a ]$
  \begin{flushleft}
    where  $\send{a} : \non A, a : A$, $\mathcal{C}[a\, u], \mathcal{D}$ are simply typed
    $\lambda$-terms and $\mathcal{C}, \mathcal{D}$ simple contexts;
    $\sq{y}$ is the (non-empty) sequence of the free variables of $u$
    which are bound in $\mathcal{C}[a\,u]$; $B$ is the conjunction of
    the types of the variables in $\sq{y}$; the displayed occurrence of $\send{a}$ is rightmost; $b$ is
    fresh and $\send{b} : \non B, b : B$
  \end{flushleft}

}
\hrule
\caption{Reduction Rules for $\lambda_{em}$}\label{fig:red_em}
\vspace{-10pt}
\end{figure*}

 \begin{definition}[Multiple Substitution]\label{defi-multsubst}
Let $u$ be a proof term, $\sq{x}=x_{0}^{A_{0}}, \ldots, x_{n}^{A_{n}}$ a sequence of  variables and $v: A_{0}\land \ldots \land A_{n}$. The substitution 
$u^{v/ \sq x}:=u[v\,\proj_{0}/x_{0}^{A_{0}} \ldots \,v\,\proj_{n}/x_{n}^{A_{n}}]$
replaces each variable $x_{i}^{A_{i}}$ of
any term $u$ with the $i$th projection of $v$.
\end{definition}

We define below the notion of stack, corresponding to Krivine's stack~\cite{Krivine} and known as \emph{continuation} because it embodies a
series of tasks that wait to be carried out. A stack represents, from
the logical perspective, a series of elimination rules; from the
$\lambda$-calculus perspective, a series of either operations
or arguments.

\begin{definition}[Stack]\label{definition-stack}
A \textbf{stack} is a, possibly empty, sequence \mbox{$\sigma =
\sigma_{1}\sigma_{2}\ldots \sigma_{n} $} such that for every $ 1\leq
i\leq n$, exactly one of the following holds: either $\sigma_{i}=t$,
with $t$ proof term or $\sigma_{i}=\pi_{j}$ with $j\in\{0,1\}$, or
$\sigma_{i}=[x_{1}.u_{1}, x_{2}.v_{2}]$ or $\exfalso_{P}$ for some
atom $P$.  We will denote the \emph{empty sequence} with $\epsilon$
and with $\xi, \xi', \ldots$ the stacks of length $1$. If $t$ is a
proof term, $t\, \sigma$ denotes the term $(((t\,
\sigma_{1})\,\sigma_{2})\ldots \sigma_{n})$.
 \end{definition}

  \begin{definition}[Case-free] A stack $\sigma$ is \textbf{case-free}
if it does not contain any sub-stack of length $1$ of the form
$[z_1.w_1, z_2.w_2]$.
  \end{definition}


%
%

\begin{definition}[Parallel Form~\cite{ACGlics2017}]\label{definition-parallel-form}
A term $t$ is a \textbf{parallel form}
  whenever, removing the parentheses, it can be written as
$$t = t_{1}\parallel_{a_{1}} t_{2}\parallel_{a_{2}}\ldots \parallel_{a_{n}} t_{n+1}$$
where each $t_{i}$, for $1\leq i\leq n+1$, is a simply typed $\lambda$-term.
  \end{definition}

\subsection{Reduction Rules and Normalization}

To present our reduction rules, we first need to define the notion of
value, which in our framework will represent a message which is ready
to be transmitted and can generate new computations when received. We
consider $a\sigma$ a value as well, when $a$ is active, since we know
that some process will eventually replace that channel endpoint with a
value.
  

\begin{definition}[Value]\label{def:value}
A \textbf{value} is a term of the form $ \langle t_1, \ldots , t_n
\rangle $, where for some $1 \leq i \leq n $, $t_i =
\lambda x \, s$, $t_i = \inj _i (s)$, $t_i = t\, \exfalso_P $, $t_{i}=t
[x.u,y.v]$ or $t_{i}=a\sigma$ for an
active channel $a$.  
\end{definition}

Communication between two processes will take place only when the corresponding channel is active.

\begin{definition}[Active Channels and Active Sessions]
We assume that the set of channel variables is partitioned into two disjoint classes: \textbf{active channels} and \textbf{inactive channels}.
A term $u \parallel_a v$ is called an \textbf{active session}, whenever $a$ is
active.
\end{definition}

A channel can be activated whenever it is applied to at least one value, thus signaling the need for communication.

\begin{definition}[Activable Channel]
  We say that $a$ is an
 \textbf{ activable channel in $u$} is there is an occurrence $av$ in $u$ for
  some value $v$.
\end{definition}

The reduction rules of $\lamem$
are in  Figure~\ref{fig:red_em}. They consist of the familiar reductions
for simply typed $\lambda$-calculus, instances of $\vel$
permutations adapted to the $\p$ operator, together with new communication reductions. 

\noindent \textbf{Activation reductions} $\quad   u \parallel_{a} v \; \mapsto \; u
[b/a] \parallel_{b} v [b/a]$ \\can be fired whenever some occurrence of
$a$ in $u$ or $v$ is applied to a value: the channel $a$ is just
renamed to an active channel $b$. This operation has the effect of
making the reduct an active session, that will produce a sequence of
message exchanges, ending only when all channel occurrences of $b$
will have transmitted their messages, regardless they are values or
not.

\noindent \textbf{Basic Cross reductions}  $\quad \mathcal{C}[\send{a}\, u]\parallel_{a} \mathcal{D} \ \mapsto \
  \mathcal{D}[ u /a ]   $ \\ can be fired whenever the channel $a$ is active and $u$ is a closed $\lambda$-term. In this case $u: A$ represents executable code or data and can replace directly all occurrences of the endpoint channel $a: A$.  

\textbf{Cross Reductions}\\
Surprisingly, logical types enable us to send open $\lambda$-terms as
messages and to fill their free variables later, when they will be
instantiated. Subject reduction will nevertheless be guaranteed. 
To see how it is possible, suppose that a channel $a$ is
used to send an arbitrary sub-process $u$ from a process $s$ to a
process $t$:
\begin{center}
\includegraphics[width=0.25\textwidth]{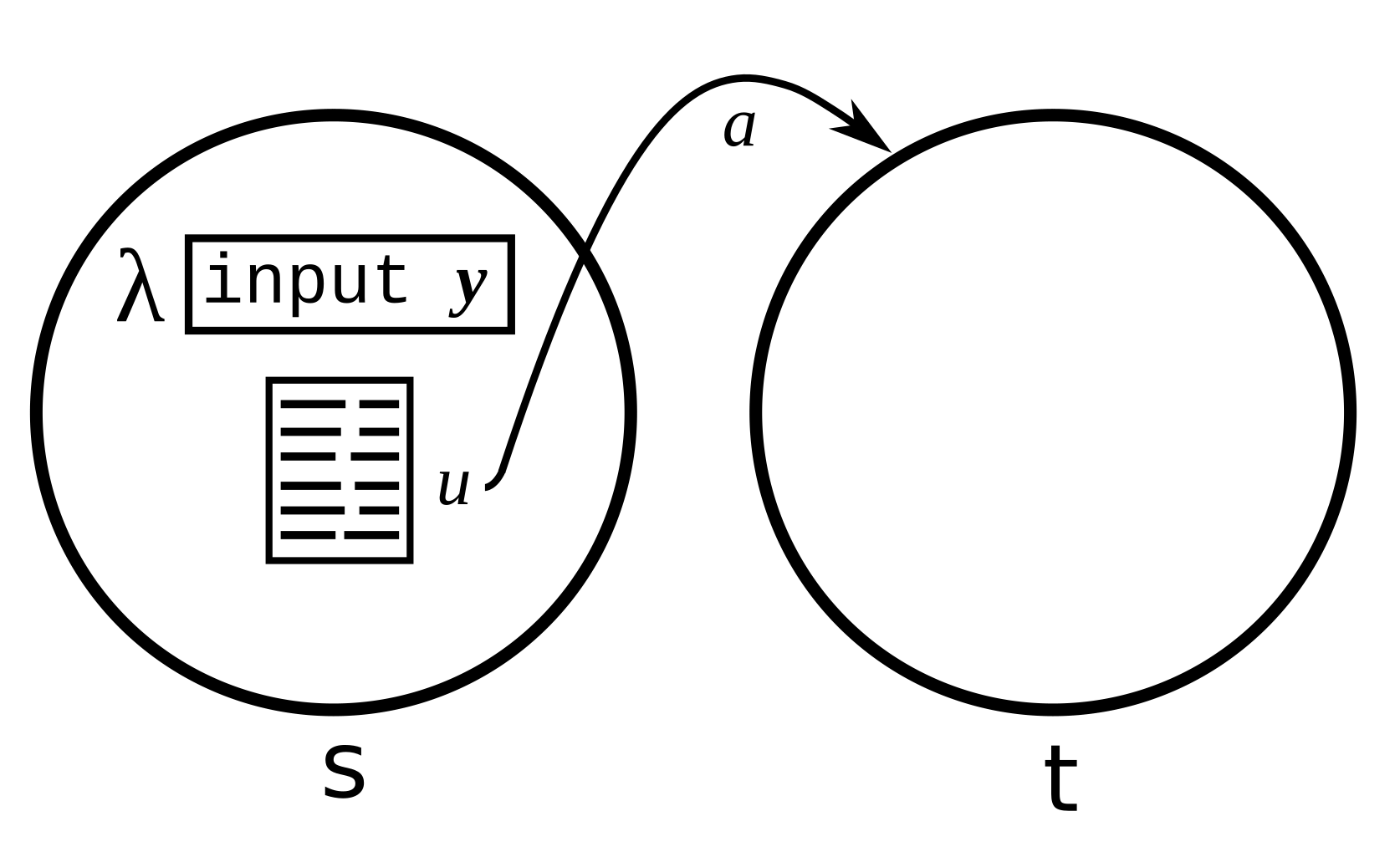}
\end{center}
Since $u$ is not necessarily closed, it might depend
on its environment $s$ for some of its inputs -- $\sq{y}$ in the
example is bound by a $\lambda$ outside $u$ -- but this is solved during the full cross reduction by a
fresh channel $b$ which redirects $\sq{y}$ to the new location of $u$: 
\begin{center}
\includegraphics[width=0.40\textwidth]{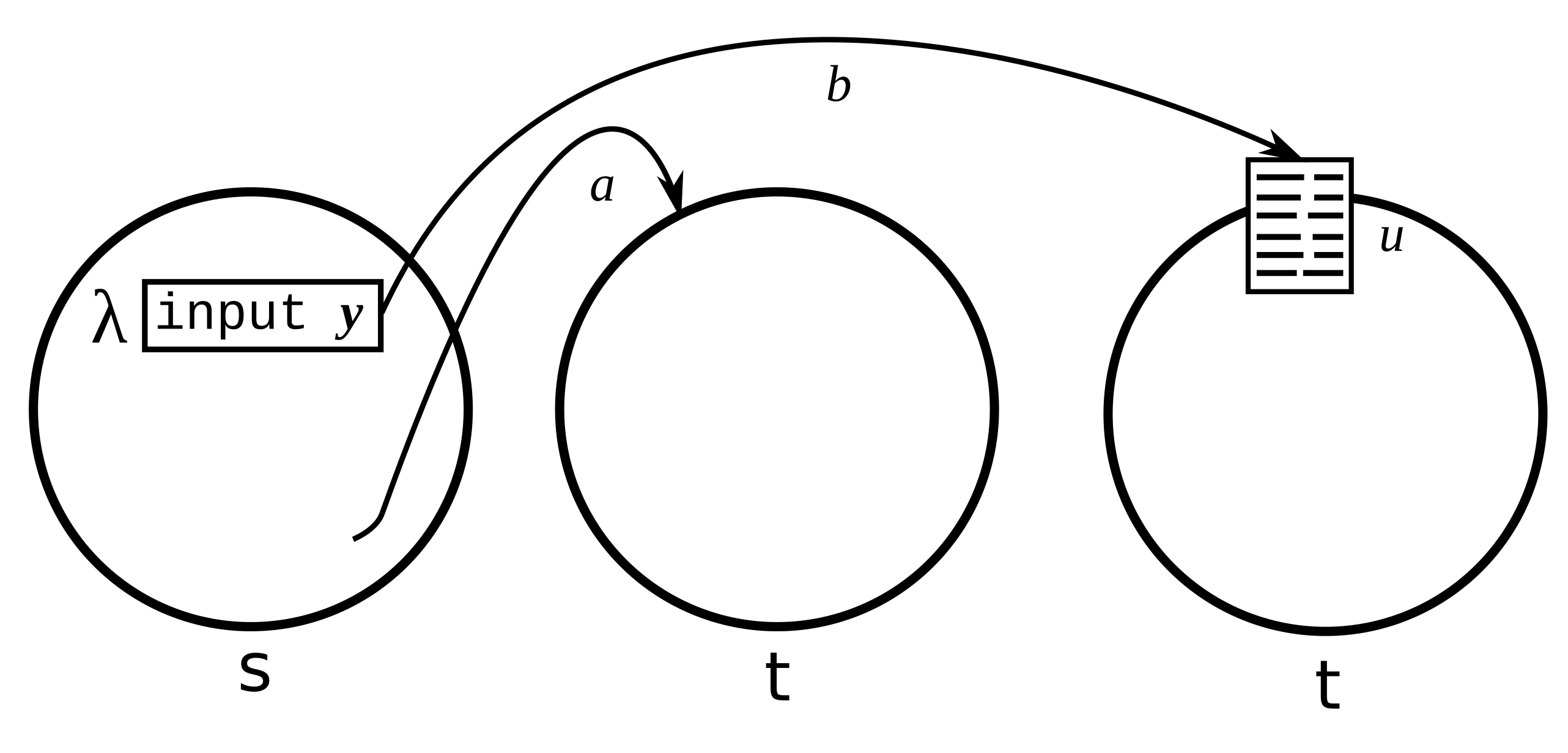}
\end{center} The old channel $a$ is kept for future messages 
$s$ might send to $t$. Technically, the cross reduction has this shape:

$$\mathcal{C}[\send{a}\, u]\parallel_{a} \mathcal{D} \ \mapsto \
    ( \mathcal{C}[\send{b} \, \lan \sq{y} \ran ] \parallel_{a} \mathcal{D} )
    \, \parallel_{b}\, \mathcal{D}[ u^{b / \sq{y}  } /a ]$$
 We see that, on one hand, the open term $u$ is replaced in
$\mathcal{C}$ by the new channel $\send{b}$ applied to the sequence
$\sq{y}$ of the free variables of $u$; on the other hand, $u$ is sent
to $\mathcal{D}$ as $u^{b / \sq{y}}$, so its free variables are
removed and replaced by the endpoint channel $b$ that will receive
their future instantiation. In Example \ref{ex:optimize} we exemplify
how to use cross reduction for program optimization.
 
 \noindent\textbf{Communication Permutations} $\; $ The only permutations for $\p$ that are not $\vel$-permutation-like
are $(u\parallel_{a} v)\parallel_{b} w \mapsto (u\parallel_{b}
w)\parallel_{a} (v\parallel_{b} w) $ and $w \parallel_{b}
(u\parallel_{a} v) \mapsto (w\parallel_{b} u)\parallel_{a}
(w\parallel_{b} v)$. These kind of permutations are between parallel operators themselves
and address the \emph{scope extrusion} issue of private channels
\cite{Milner}. A parallel operator is allowed to commute with another
one only when strictly necessary, that is, if there is not already an
active session inside $u, v, w$ that can be first reduced.


The idea behind the normalization (see Section~\ref{sec:norm} for the proof) is to
apply
to
any term $t$ the following recursive procedure:
 \begin{enumerate}
 \item \emph{Parallel normal form production}. We transform $t$ into a
term $u$ in parallel normal form, using permutations
(Prop.~\ref{proposition-parallelform}).
 
 \item \emph{Intuitionistic Phase}. As long as $u$ contains intuitionistic redexes, we apply intuitionistic reductions.
 
\item \emph{Activation Phase}. As long as $u$ contains activation redexes, we apply activation reductions.

\item \emph{Communication Phase}. As long as $u$ contains active
  sessions, we select the uppermost, permute it upward if necessary or
  apply directly a cross reduction. 
  Go to step 2.
 
\end{enumerate}
 
  Since we are dealing with a Curry-Howard isomorphism, every reduction rule of $\lamem$ corresponds to a reduction for the natural deduction calculus $\NJ + (\exm)$.

Basic reductions correspond to the logical reductions:

\[ \vcenter{\infer{F}{\infer*{F}{\infer{
\fal}{\deduce{A}{\delta}}}&&\infer*{F}{[A]}} }\quad \mapsto
\quad  \vcenter{ \infer*{F}{\deduce{A}{\delta}} }\] 
where no assumption in $\delta$ is discharged below $\bot$ and above $F$. When this is the
case, intuitively, the displayed instance of $(\exm )$ is hiding some
redex that should be reduced. The reduction precisely expose this
potential redex and we are  thus able to reduce it.



Cross reductions correspond to the logical reductions:

\[ \vcenter{\infer{F}{\infer*{F}{\infer{
\fal}{\deduce{A}{\deduce{\delta}{\Gamma}}}}&&\infer*{F}{[A]}} }\quad \mapsto \vcenter{\infer[*]{
F}{ \infer{F}{\infer*{
F}{\infer[*]{ \fal}{\infer=[\et i]{\bigwedge \Gamma }{ \Gamma}
}}&&\infer*{F}{[A]}} &&\infer*{F}{\deduce{A}{\deduce {\delta}{\infer=[\et
e]{\Gamma}{[ \bigwedge \Gamma]^* }}}}}}\] Here we use $\delta$ to
prove all assumptions $A$, as before, but now we also need to
discharge the assumptions $\Gamma $ open in $\delta$ and discharged below $\bot$ and above $F$ in
the rightmost branch. This is done by a new $( \exm )$ rule application
$*$ to the conjunction $\bigwedge\Gamma $ of such
assumptions. Accordingly, we use the inference
$\vcenter{\infer{ \fal}{
\bigwedge \Gamma}}$ to replace one occurrence of
$\vcenter{\infer{ \fal}{A}}$ in the leftmost
branch. In order to justify the other occurrences of the latter rule,
we need the $( \exm ) $ application joining the leftmost branch and the
central branch, which is just a duplicate.

\subsection{Computing with $\lamem$}
We present few examples of the computational capabilities of $\lamem$. 
%

\begin{example}[\textbf{$\lamem$ vs the 
simply typed $\lambda$-calculus}]  
As is well known there
is no $\lambda$-term $\mathsf{O}: \mathsf{Bool} \IMPL \mathsf{Bool}
\IMPL\mathsf{Bool} $ such that $\mathsf{O}\mathsf{F}\mathsf{F} \mapsto \mathsf{F}$, $\mathsf{O}u\mathsf{T} \mapsto \mathsf{T}$,
$\mathsf{O}\mathsf{T}u \mapsto \mathsf{T}$,   
where $u$ is, e.g., a variable.  $\mathsf{O}$ can instead be defined in
Boudol's calculus~\cite{Boudol89} and in $\lamg$~\cite{ACGlics2017}. Assuming to add the boolean type in our calculus, the $\lam _{em}$ term for such parallel OR is
\[\mathsf{O} :=  \lambda x^{\mathsf{Bool}} \, \lambda
    y^{\mathsf{Bool}} (P_1 \parallel_a P_2) \]
where ($\send{a}: \mathsf{Bool}\IMPL \bot$, $a: \mathsf{Bool}$ and the construct ``$\mathsf{if} \, u \, \mathsf{then} \, s \,
\mathsf{else} \, t$'' is as usual)
\[P_1 = \mathsf{if} \, x \, \mathsf{then} \, \mathsf{T} \, \mathsf{else}
\,   \efq{\mathsf{Bool}}{\send{a}  \mathsf{F}} \quad \; \mbox{and} \; \quad 
P_2 = \mathsf{if} \, y \,
\mathsf{then} \, \mathsf{T} \, \mathsf{else} \, a
\]
\begin{small}\begin{align*}
\mbox{Now} \quad \mathsf{O}\, u \,\mathsf{T}     &\mapsto^{*}  
(\mathsf{if} \, u \, \mathsf{then} \, \mathsf{T} \, \mathsf{else}
\,  \efq{\mathsf{Bool}}{\send{a}  \mathsf{F}})
    \parallel_{a} (\mathsf{if} \, \mathsf{T} \,
\mathsf{then} \, \mathsf{T} \, \mathsf{else} \, a)\\
&\mapsto^{*} (\mathsf{if} \, u \, \mathsf{then} \, \mathsf{T} \, \mathsf{else}
\, \efq{\mathsf{Bool}}{\send{a}  \mathsf{F}})  \parallel_{a} \mathsf{T} \,  \mapsto \,
    \mathsf{T}
  \end{align*}
\end{small} And symmetrically $\mathsf{O}\,  \mathsf{T}\, u \,
\mapsto^{*} \, \mathsf{T} $. 
On the other hand,\begin{small}
  \begin{align*}
\mathsf{O} \, \mathsf{F} \,\mathsf{F}  \mapsto^{*} &
(\mathsf{if} \, \mathsf{F} \, \mathsf{then} \, \mathsf{F} \, \mathsf{else}
\,  \efq{\mathsf{Bool}}{\send{a}  \mathsf{F}})    \parallel_{a} (\mathsf{if} \, \mathsf{F} \,
\mathsf{then} \, \mathsf{F} \, \mathsf{else} \, a) \mapsto^{*}    \efq{\mathsf{Bool}}{\send{a}  \mathsf{F}}
    \parallel_{a} a
\; \mapsto \;
     \mathsf{F}
  \end{align*}\end{small}
\end{example}

\begin{example}[\textbf{Cross reductions for program efficiency}]
\label{ex:optimize} 
Consider the parallel processes: $M \parallel_{d} (
P \parallel_{a} Q)$. The process $P$ contains a channel $a$ to send
the message $\lan s , y \ran$ to $Q$. But the variable $y$ stands for a
missing part of the message which needs to be replaced by a term that
$M$ has to compute and send to $P$. Hence the whole interaction needs to 
wait for $M$. The cross reduction
handles precisely this kind of missing arguments. It
enables $P$ to send immediately the message through the channel $a$
and establishes a new communication channel on
the fly which redirects the missing term, when ready, to the new location of the message inside $Q$.
As a concrete example assume that\begin{small}
$$    M \mapsto^{*}\; {\mathcal M} \, [\send{d} \, (\lambda x \, xt )] \quad
    P  := \quad   {\mathcal P}\, [d (\lam y \, \send{a} \lan s ,
        y \ran )]
          \quad 
    Q  : = \quad  {\mathcal Q}\, [a\pi_0]$$
\end{small}where $s$ and $t$ are closed terms, and the contexts
${\mathcal M}, {\mathcal P}$ and ${\mathcal Q}$ do not contain
other instances of the channels $d$ and $a$. Without a special mechanism for
sending open terms, $P$ must wait for $M$ to normalize. Afterwards $M$
passes $\lambda x \, xt$ through $d$ to $P$ by the following
computation:
  \begin{align*} M \parallel_{d} ( P \parallel_{a} Q) \,  \mapsto^{*}
 \, 
(M \parallel_{d} P) \parallel_{a} Q \, \mapsto^{*} ({\mathcal M} \, [\send{d} \, (\lambda x \, xt )] \parallel_{d} P) \parallel_{a} Q \\ 
\mapsto {\mathcal P}\, [( \lambda x \, xt) (\lam y \, \send{a} \lan s ,
        y \ran )] \p_a
    Q 
\mapsto {\mathcal P}\, [(\lam y \, \send{a} \lan s ,
        y \ran ) t] \p_a
     Q \,  \mapsto \\
{\mathcal P}\, [\send{a} \lan s ,
        t \ran ] \p_a
    {\mathcal Q}\, [a\pi_0] \, \mapsto \, {\mathcal Q}\, [\lan s,
                                 t\ran \pi_0] \, \mapsto  {\mathcal Q}\, [s] 
  \end{align*}
But it is clear that $Q$ does not need $t$ at all. Even though $Q$
waited so long for the pair $\lan s, t\ran $, it only keeps
the term $s$. 

Our normalization algorithm allows instead  $P$ to directly send
$\lambda y \, \send{a} \lan s, y\ran  $ to $Q$ by executing a full cross reduction:
\begin{align*}
M \p_d (   {\mathcal
  P}\, [d (\lam y \, \send{a} \lan s , y\ran )] \p_a  {\mathcal Q}\,
  [a\pi _0])\;  \mapsto^{*}\\ 
    M \p_d ( (   {\mathcal P}\, [d (\lam y \, \lan s , \send{b} y\ran
  )] \p_a Q) \p _b   {\mathcal Q}\, [\lan s , b \ran \pi _0] )
\;  \mapsto^{*}\\
M \p_d ( (   {\mathcal P}\, [d (\lam y \, \lan s , \send{b} y\ran
  )] \p_a Q) \p _b  {\mathcal Q}\, [\lan s , b \ran \pi _0] )
\end{align*}
where the communication $b$ handles the redirection of the data $y$ in
case it is available later. But in our case $Q$
already contains all it needs to terminate its computation, indeed 
\begin{align*}
\mapsto M \p_d ( (   {\mathcal P}\, [d (\lam y \, \lan s , \send{b} y\ran
  )] \p_a Q) \p _b   {\mathcal Q}\, [s]) \mapsto^*  {\mathcal Q}\, [s]
\end{align*}since ${\mathcal Q}\, [s]$ does not contain communications
anymore. Notice that the time-consuming normalization of the term $M$ does not
need to be finished at this point.
%
\end{example}

\subsection{Close relatives of the calculus}\label{sec:broad}

As often happens with computational interpretations of logics --
compare, for example~\cite{Griffin} and~\cite{Parigot} -- logical
equivalence of formulae here does not play a decisive r\^ole.  When
extracting communications mechanisms from disjunctive tautologies,
differences in the proof-theoretical representation of the axiom can
lead to essential differences in the computational behavior of the
resulting communication mechanisms. A particular striking case is the
class of formulae with the following shape: $\non A \vel A \vel \dots
\vel A$. Indeed, even though all these formulae are equivalent in a
very strong sense to $(\exm)$, it is very interesting to consider
their computational counterparts as distinct objects. Its typing rule
is:
\[\infer{\pp{a}{u \p v_1\p \dots \p v_m} :
F}{\infer*{u:F}{ [\send{a} : \non A]} & \infer*{v_1 : F}{[a: A]} & \dots &
\infer*{v_n : F}{[a: A]} }\]
The associated  basic cross reduction leads to a broadcasting system:
  \[\pp{a}{\mathcal{C}[\send{a}\, u]\parallel \mathcal{D}_1 \dots \parallel \mathcal{D}_n} \;  \mapsto \;  \ \mathcal{D}_1 [ u /a ] \p \dots \parallel \mathcal{D}_n[ u /a ]\] where $\send{a} : \non A, a : A $,  $\mathcal{C}[a\, u]$ are simply typed
$\lambda$-terms, the
sequence of free variables of $u$ is empty,
$a$ does not occur in $u$, and $\p$ is the operator typed by the contraction rule. The reduction implements a
communication schema in which the process corresponding to $\non A $
broadcasts the same message to the processes corresponding the different
instances of $A$.

%
%

\section{General case: the $\lama$ calculi}
\label{sec:ax}

We introduce the $\lama$ calculi that solve the equation
$$\infer{\NJ}{\mbox{simply typed } \lambda\mbox{-calculus}} \quad = \quad  
\infer{\NAx}{?}$$
To simplify notation (see remark \ref{rem:cond}), we shall only consider  $$\ax = (A_1  \impl B_1)\vel \dots \vel (A_m \impl B_m)$$ where
 no $A_i$ is repeated and for every  $A_{i}\neq \top$, $A_{i}=B_{j}$ for some $j$. 
%

The set of type assignment rules for $\lama$ terms are obtained
replacing the last three rules in Table~\ref{tab:type_em}
by
\smallskip
\hrule 
\begin{footnotesize}
  \begin{equation}\label{ourN}
\vcenter{\infer{a^{A \impl B } \, u :
         B}{ u : A}}\qquad \vcenter{\infer[(\ax )]{\pp{a}{u_{1} \p\dots\p u_{m}}:
         B}{\infer*{u_{1}:B}{[a^{\scriptscriptstyle A_{1}\IMPL B_{1}}:
           A_{1}\IMPL B_{1}]}&\dots &\infer*{u_{m}:B}{[a^{A_{m}\IMPL
             B_{m}}: A_{m}\IMPL B_{m}]}}}
   \end{equation}
\end{footnotesize}
all occurrences of $a$ in $u_{i}$ for $1\leq i\leq m$ are of
the form $a^{A_{i}\IMPL B_{i}}$.
\smallskip
\hrule 




    




\begin{figure*}[!htb]
\hrule
\begin{footnotesize}

\smallskip
\begin{flushleft}
  \textbf{Intuitionistic Reductions} \hspace{20pt}
  $(\lambda x^{\scriptscriptstyle A}\, u)t\mapsto
  u[t/x^{\scriptscriptstyle A}] \qquad
  \pair{u_0}{u_1}\,\pi_{i}\mapsto u_i, \mbox{ for $i=0,1$} \qquad 
  $ $ \inj _i (t) [x_0.u_0,x_1.u_1] \mapsto u_i [t/x_i]$
\end{flushleft}

\smallskip
 
\begin{flushleft}
  \textbf{Disjunction Permutations} \hspace{70pt} $t [x_0.u_0,x_1.u_1]
\xi \mapsto t [x_0.u_0\xi,x_1.u_1\xi] \text{ if }\xi\text{ is a
one-element stack}$
  \end{flushleft}

\begin{flushleft}
  \textbf{Parallel Operator Permutations }  $\pp{a}{u_{1}\p\dots
    \p u_{m}}\, \xi \mapsto  \pp{a} {u_{1} \xi\p \dots\p
     u_{m} \xi },  \text{ if  $\xi$ is a
one-element stack and  $a$ does
    not occur free in $\xi$}$ 
\end{flushleft}
 
 $w \pp{a}{u_{1}\p \dots \p u_{m}} \mapsto
 \pp{a}{ w u_{1}\p\dots \p w u_{m}},  \mbox{ if $a$ does
    not occur free in $w$} $


 $\lambda x^{\scriptscriptstyle A}\,\pp{a}{u_{1}\p\dots \p u_{m}} \mapsto
\pp{a}{\lambda x^{\scriptscriptstyle A} u_{1}\p\dots \p \lambda x^{\scriptscriptstyle A} u_{m}}$

$\inj_{i}(\,\pp{a}{u_{1}\p\dots \p u_{m}}) \mapsto
\pp{a}{\inj_{i}(u_{1})\p\dots \p \inj_{i}(u_{m})}$

$\langle \pp{a}{ u_{1}\p \dots\p u_{m}},\, w \rangle \mapsto \pp{a}{
  \langle u_{1}, w\rangle\p \dots\p \lan u_{m}, w\ran }, \mbox{ if $a$ does
    not occur free in $w$}$

$\langle w, \, \pp{a}{u_{1}\ \dots\p u_{m}}\rangle
\mapsto \pp{a}{ \lan w, u_{1}\ran\p \dots\p  \lan w, u_{m} \ran } ,\mbox{ if $a$ does
    not occur free in $w$}$

 $ \pp{a}{ u_1\p \dots \p \pp{b}{w_1\p  \dots \p w_n}  \dots \p
u_m} \; \mapsto \;   \pp{b}{ \pp{a}{ u_1\p \dots \p w_1\dots \p u_m}
\dots \p  \pp{a}{ u_1\p \dots \p w_n \dots \p u_m}}  $

\mbox{ if $u_1 , \dots , u_m$ and $\pp{b}{w_1\p  \dots \p w_n}$ do not contain active sessions}










\medskip

\textbf{Communication Reductions}

\begin{flushleft} \textbf{Activation Reduction} \hspace{75pt} 
$ \pp{a}{ u_{1}\p \dots \p u_m} \;
\mapsto \; \pp{a}{ u_1 [b / a]\p \dots
\p u_{m} [b/a]}$

where $a$ is not active, $b$ is a fresh \emph{active} variable, and
there is some occurrence of $a$ in $u$ or in $v$  of the form $a
w$, for a value $w$.
\end{flushleft}

\smallskip

  \begin{flushleft}
    \textbf{Basic Cross Reductions }  $\pp{a} { {\mathcal C}_1 \p \dots {\mathcal C}_i[a^{F _i \impl
G_i } \, t]  \p \dots \p {\mathcal C}_j[a^{F _j \impl G_j } \, u] \p
\dots \p {\mathcal C}_m} \; \mapsto \; \pp{a} {{\mathcal C}_1 \p \dots
\p {\mathcal C}_i[a^{F _i \impl G_j } \, t] \p \dots {\mathcal
C}_j[t] \p \dots \p {\mathcal C}_m }\qquad$  for $F_i = G_j$
\end{flushleft}
$\pp{a} {{\mathcal C}_1 \p \dots \p {\mathcal C}_i[a^{F _i \impl G_i }
\, u ] \p \dots {\mathcal C}_j[a^{F _j \impl G_j } \, t]  \p
\dots \p {\mathcal C}_m} \; \mapsto \; \pp{a} {{\mathcal C}_1 \p \dots
 {\mathcal C}_i[t]  \p \dots \p {\mathcal C}_j[a^{F _j \impl
G_j } \, t] \p \dots \p {\mathcal C}_m} \qquad$ for $F_j = G_i$
\begin{flushleft}
where $a$ is active, the displayed occurrence of $a$ rightmost in the simply typed $\lambda$-terms ${\mathcal C}_i[a^{F _i \impl G_i }t]$ and ${\mathcal C}_j[a^{F _j \impl G_j}u]$,  $1 \leq i <j\leq m$, $\mathcal{C}_{i}, {\mathcal
  C}_j$ are simple contexts and $t$ is closed.
   \end{flushleft}

\smallskip

\begin{flushleft} \textbf{Cross Reductions} \hspace{30pt}  
$\pp{a}{ u_1 \p \dots \p u_m } \mapsto u_{j_1} \p \dots \p u_{j_n} $,
for $1 \leq j_1 <  \dots < j_n \leq m $, if $a$ does not occur in
$u_{j_1} , \dots , u_{j_n}$
\end{flushleft}

  \smallskip

$ \pp{a}{ {\mathcal C}_1 [a^{F_1 \impl G_1 } \,t_1] \p \dots \p
{\mathcal C}_m[a^{F_m \impl G_m } \,t_m]} \; \mapsto \; \pp{b}{ s_1 \p
\dots \p s_m } $
\begin{flushleft} where  $a$ is active, ${\mathcal C}_j [a^{ F_j \impl G_j } \,t_j] $ for $1 \leq j \leq
m $ are simply typed $\lambda$-terms; the displayed $a^{ F_j \impl G_j } $ is
rightmost in each of them;
$b$ is fresh; for $1 \leq i \leq m $, we define, if $G_i \neq \fal$,
\[ s_{i} \; = \; 
\pp{a} { {\mathcal C}_1 [a ^{F_1 \impl G_1}\, t_1] \dots \p  {\mathcal C }
_i [ t_j^{b_i \lan \sq{y}_i \ran  /
\sq{y}_j} ] \p \dots \p {\mathcal C}_m[a^{F_m \impl G_m} \,t_m]}
\quad   \text{ and if } \; G_i = \fal \quad  s_{i} \; = \; 
\pp{a}{ {\mathcal C}_1 [a^{G_1 \impl G_1}\, t_1] \dots \p {\mathcal
    C}_i [b_i \lan \sq{y}_i \ran ] \p \dots \p 
{\mathcal C}_m [a^{ F_m \impl G_m }\,t_m]}\]
where $F_j = G_i$; $\sq{y}_z$ for $1 \leq z \leq m$ is the sequence of the free
variables of $t_z$ bound in $\mathcal{C}_z[a^{F_z \impl G_z}
\, t_z]$; $b_i=b^{B_{i}\IMPL B_{j}}$, where $B_{z}$ for $1 \leq z \leq
m$ is the type of $\lan \sq{y}_z \ran$.
\end{flushleft}



\end{footnotesize}
\hrule
\caption{Reduction Rules for $\lama$ with 
$\ax=(F_{1}\rightarrow G_{1})\lor \dots \lor (F_{k}\rightarrow G_{k})
$}\label{fig:red}
\vspace{-10pt}
\end{figure*}

\subsection{Reduction rules of $\lama$}

Fig.~\ref{fig:red} below shows the reductions for $\lama$-terms.
The permutations rules are simple generalisations of those of
$\lamem$ which only 
adapt the latter to $n$-ary
parallelism operators. On the other hand, $\lama$ cross reductions
are more complicated than those of 
$\lamem$.
The formula $A$ in $A \vel \non A$ corresponds indeed to a process that can only
receive communications from other processes.
The $\lama$ cross reductions cannot rely in all cases on these
features and hence require a more general formulation.

\noindent\textbf{Basic cross reductions} $\;$ They implement a simple  communication of closed programs $t$
\begin{align*}
&\pp{a} { {\mathcal C}_1 \p \dots {\mathcal C}_i[a^{F _i \impl
G_i } \, t]  \p \dots \p {\mathcal C}_j[a^{F _j \impl G_j } \, u] \p
\dots \p {\mathcal C}_m} \; \mapsto \\ 
& \pp{a} {{\mathcal C}_1 \p \dots
\p {\mathcal C}_i[a^{F _i \impl G_j } \, t] \p \dots {\mathcal
C}_j[t] \p \dots \p {\mathcal C}_m }
\end{align*}
While the sender ${\mathcal C}_i [a_i t ]$ is
unchanged, ${\mathcal C}_j [a_j\, u] $ receives the message and
becomes ${\mathcal C}_j [t]$. Here, unlike in $\lamem$ cross
reductions, only one receiving channel can be used for each
communication. Indeed, if some occurrences of $a_{i}$ are nested, a
communication using all of them would break subject reduction.

\noindent \textbf{Cross Reductions} $\; $ Basic cross reductions allow
non deterministic closed message passing. As in classical logic, cross
reductions  implement 
communication with an
additional mechanism
for handling migrations of open processes.
But here each cross reduction application combines at once all
possible message exchanges 
implemented by 
basic cross reductions, since one cannot
know in advance which process will become closed and actually be sent.

For a proof-theoretic view, consider 
the application of $(\ax )$ (below left), in which all 
 $\Gamma_i $ for $1 \leq i \leq m$ are discharged between $B_i$ and
$F$. It reduces by full cross reduction to the derivation below right (we explicitly mark with the same label the rule applications
belonging to the same higher-level rule):
  \[\vcenter{\infer[*]{F}{ \infer*{F}{\infer[*]{B_1}{\deduce{A_1}{\deduce{\alpha
              _1}{\Gamma_1}}}} & \dots
      \infer*{F}{\infer[*]{B_m}{\deduce{A_m}{\deduce{\alpha
              _m}{\Gamma_m}}}}}} \qquad  \mapsto \qquad   \vcenter{\infer[**]{F}{\delta _1 & \dots &
\delta _m}}\]
%
%
such that for $ 1 \leq i \leq m$ the derivation  $ \delta _i$ is 
\begin{footnotesize}
  \[ \infer[*]{F}{\deduce{F} {\vspace{3pt}\dots }& \dots
&\infer*{F}{\deduce{B_i}{\infer=[\et
e]{\vspace{3pt} \alpha_j}{\infer[**]{ \bigwedge\Gamma
_j}{ \infer=[\et i]{ \bigwedge\Gamma _i}{\Gamma
_i }}} }}& \dots & \deduce{ F}
{\vspace{3pt}\dots }}\]
\end{footnotesize}in which $\alpha_j $ is the derivation of the
premiss $A_j = B_i$ associated by $\ax$ to $B_i$ and a double
inference line denotes a derivation of its conclusion using the named
rule possibly many times.


\begin{theorem}[Subject Reduction]\label{subjectred}
If $t : A$ and $t \mapsto u$, then $u : A$ and all the free variables of $u$ appear among those of $t$.
\end{theorem} 
\begin{proof} 
It is enough to prove the theorem for the cross reductions: if $ t : A$
and $t \mapsto u$, then $u : A$. The proof that the intuitionistic
reductions and the permutation rules preserve the type is completely
standard. Basic cross reductions require straightforward considerations as well.
Suppose then that
\[ \pp{a} {{\mathcal C}_1 [a^{A_{1}\rightarrow B_{1}}\,t_1] \p \dots \p {\mathcal
C}_m[a^{A_{m}\rightarrow B_{m}}\,t_m]} \; \mapsto \; \pp{b} { s_1 \p \dots \p  s_m  } \]
where $s_i$ for $1 \leq i \leq m $ is
\begin{footnotesize}
  \[\pp{a} {{\mathcal C}_1 [a^{A _1 \impl B_1 }\, t_1] \p \dots \p
    {\mathcal C } _i[ t_j^{b^{B_i \impl B_j} \lan \sq{y}_i \ran /
      \sq{y}_j} ] \p \dots \p {\mathcal C}_m[a^{A _m \impl B_m
    }\,t_m]} \]
\end{footnotesize}Since $ \langle \sq{y}_i \rangle: B_i $, $b^{B_i
\impl B_j}$ are correct terms, and hence $ t_j^{b^{B_i \impl B_j} \lan
\sq{y}_i \ran / \sq{y}_j} $, by Definition ~\ref{defi-multsubst}, are
correct as well. The assumptions are that $t_j^{b^{B_i \impl B_j} \lan
\sq{y}_i\ran / \sq{y}_j}$ has the same type as $a^{A _i \impl B_i }\,
t_i$; $\sq{y}_i $ for $1 \leq i \leq m$ is the sequence of the free
variables of $t_i$ which are bound in
$\mathcal{C}_i[a^{A_{i}\rightarrow B_{i}}\, t_i]$; $B_i$ for $1 \leq i
\leq m$ is the conjunction of the types of the variables in
$\sq{y}_i$; $a$ is rightmost in each ${\mathcal C}_i [a^{A _i \impl
B_i }\,t_i] $; and $b$ is fresh. Hence, by construction all the
variables $\sq{y}_i$ are bound in each ${\mathcal C } _i[ t_j^{b^{B_i
\impl B_j} \lan \sq{y}_i \ran / \sq{y}_j} ]$. Hence, no new free
variable is created.
\end{proof}

\begin{example}[G\"odel Logic]
A particular instance of $\lama$ is defined by $\ax = (A
\impl B ) \vel (B\impl A)$. The resulting type assignment rules and
reductions are those of $\lamg$~\cite{ACGlics2017} but for the
activation conditions of these, which are no more based on the types.
\end{example}

\noindent For  the Subformula Property, we shall need the following definition. 
\begin{definition}[Prime Formulas and Factors \cite{Krivine1}]
  A formula is said to be \textbf{prime} if it is not a
  conjunction. Every formula is 
  a conjunction of prime
  formulas, called \textbf{prime factors}. 
\end{definition}

 \begin{definition}[Normal Forms and Normalizable Terms]\mbox{}
 \begin{itemize}
\item  A \textbf{redex} is a term $u$ such that $u\mapsto v$ for some $v$ and basic reduction of Figure \ref{fig:red}. A term $t$ is called a \textbf{normal form} or, simply, \textbf{normal}, if there is no $t'$ such that $t\mapsto t'$. We define $\nf$ to be the set of normal $\lama$-terms. 
\item  
A sequence, finite or infinite, of proof terms
$u_1,u_2,\ldots,u_n,\ldots$ is said to be a reduction of $t$, if
$t=u_1$, and for all  $i$, $u_i \mapsto
u_{i+1}$.
 A proof term $u$ of $\lama$ is  \textbf{normalizable} if there is a finite reduction of $u$ whose last term is a normal form. 
\end{itemize}
\end{definition}

\begin{definition}[Parallel Form]\label{definition-parallel-form} A
term $t$ is a \textbf{parallel form} whenever, removing the
parentheses, it can be written as
$$t = t_{1}\parallel  t_{2}\parallel \ldots \parallel  t_{n+1}$$
where each $t_{i}$, for $1\leq i\leq n+1$, is a simply typed
$\lambda$-term.
\end{definition}

\begin{definition}[Active Channels and Active Sessions] We assume that
the set of channel variables is partitioned into two disjoint classes:
\textbf{active channels} and \textbf{inactive channels}.  A term
$\pp{a}{u_1 \p \dots \p u_m }$ is called an \textbf{active session},
whenever $a$ is active.
\end{definition}

\subsection{Computing with $\lama$}

To limit the non-determinism of  $\lama$  basic cross reductions,
%
we impose that 
only an underlined term $\chosen{{\mathcal C}_i [a^{F _i \impl
G_i} \, t ]}$ (respectively $\chosen{{\mathcal C}_j [a^{F _j \impl
G_j} \, t ]}$ ) can send a message. The underlining is
moved to the receiving term $\chosen{{\mathcal C}_j[t]}$ (respectively
 $\chosen{{\mathcal C}_i[t]}$).


\begin{example}[\textbf{$\lambda_{C_3}$ and Cyclic scheduling}]
\label{Milnerscheduler}

The basic cross reductions for the cyclic axiom $C_3 \; = \; (A \impl
B) \vel( B \impl C) \vel (C \impl A)$ are  the following:
\begin{align*} &
\pp{a} { \chosen{ {\mathcal C}_1 [a_1 s]} \p
{\mathcal C}_2 [a_2 t] \p {\mathcal C}_3} \mapsto \pp{a}{ {\mathcal
                 C}_1 [a_1 s] \p \chosen{ {\mathcal C}_2 [s]
} \p {\mathcal C}_3}\\ 
& \pp{a} {  {\mathcal C}_1
\p \chosen{ {\mathcal C}_2 [a_2 s] } \p_{a_3} {\mathcal C}_3[a_3 t]
} \mapsto \pp{a}{  {\mathcal C}_1  \p {\mathcal C}_2
[a_2 s] \p \chosen{ {\mathcal C}_3[t]} }\\ 
&  \pp{a}{ {\mathcal C}_1 [a_{1} t] \p {\mathcal C}_2 \p \chosen{ {\mathcal
C}_3 [a_3 s ]}} \mapsto \pp{a}{ \chosen{ {\mathcal C}_1 [s]}
\p {\mathcal C}_2  \p {\mathcal C}_3[a_3 s]}
\end{align*}
We use them to implement a cyclic scheduler along the lines of that
in~\cite{Milner_com}.  Such scheduler is designed to ensure that a
certain group of processes performs the assigned tasks in cyclic
order. Each process can perform its present task in parallel with each
other, and each of them can stop its present round of computation at
any moment, but no process should start its $n$th round of computation
before its predecessor has started the respective $n$th round of
computation.

Consider the processes
$\chosen{{\mathcal C}[a_1 ( r (a_1\, r') ) ]}$, ${\mathcal D}[a_2\, ( s (
a_2\, s' )) ] $ and \\${\mathcal E}[a_3\, ( t (a_3\, t')) ] $ where $a_1 : A
\impl B$, $a_2:  C \impl A$ and $a_3:  B \impl C$. We implement the scheduling
program as the following $\lam _{C_3}$ term
\[\pp{a}{\chosen{ {\mathcal C}[a_1\, ( r  (a_1\, r' )) ]} \p
    {\mathcal D}[a_2\, (  s (a_2\, s' )) ] \p  {\mathcal E}[a_3\, (  t (a_3\, t' )) ]}  \] When $r'$ is over with its computation, according to the
strategy in Def.~\ref{defi-mastredstrategy}, this term reduces to \[\pp{a} {  {\mathcal C}[a_1\,  (  r ( a_1\, r' )) ] \p \chosen{
      {\mathcal D}[a_2\,  (  s  r' ) ] } \p {\mathcal
    E}[a_3\,  (  t ( a_3\, t' )) ]}\]
then, after the evaluation of $s$, to
\[ \pp{a}{ {\mathcal C}[a_1\,  (  r ( a_1\, r' )) ] \p {\mathcal
  D}[a_2\,  (  s  r' ) ]  \p \chosen{ {\mathcal E}[a_3\,  (  t  (  s  r' ) ) ]}}
\]
and after the evaluation of $t$ to
\[ \pp{a}{ \chosen{  {\mathcal C}[a_1\,  (  r  (  t  (  s 
      r' ) ) ) ] } \p {\mathcal
    D}[a_2\, (   s  r' )  ] \p {\mathcal E}[a_3\,  (  t  (s r')
  ) ]}
\] and so on until all arguments of channels $a_i$ for $i \in \{1,2,3\}$ have normalized and have been communicated.

\end{example}

\section{The Normalization Theorem}\label{sec:norm}

Our goal is to prove the Normalization Theorem for $\lama$:
every proof term of $\lama$ reduces in a finite number of steps to a
normal form.  By Subject Reduction, this implies that the corresponding natural deduction proofs do
normalize. We shall define a reduction strategy for terms of $\lama$:
a recipe for selecting, in any given term, the subterm to which apply
one of our basic reductions. 

The idea behind our normalization strategy is quite intuitive. We start from any term and reduce it in parallel normal form, thanks to Proposition \ref{proposition-normpar}. Then we  cyclically interleave three reduction phases. First, an \emph{intuitionistic phase}, where we reduce all intuitionistic redexes. Second, an \emph{activation phase}, where we activate all sessions that can be activated. Third, a \emph{communication phase}, where we allow the active sessions to exchange messages as long as they need and we enable the receiving process to extract information from the messages. Technically, we perform all cross reductions combined with the generated structural intuitionistic redexes, which we consider to be projections and case permutations.

Proving termination of this strategy is by no means easy, as we have to rule out two possible infinite loops.

\begin{enumerate}
\item Intuitionistic reductions can generate new activable sessions that want to transmit messages, while message exchanges can generate new intuitionistic reductions.

\item During the communication phase, new sessions may be generated after each cross reduction and old sessions may be duplicated after each session permutation. The trouble is that each of these sessions may need to send new messages, for instance forwarding some message received from some other active session. So the count of active sessions might increase forever and the communication phase never terminate.
\end{enumerate}

We break the first loop by focusing on the complexity of the exchanged
messages. Since messages are \emph{values}, we shall define a notion
of \emph{value complexity} (Definition \ref{def:fut_comp}), which will
simultaneously ensure that: (i) after firing a non-structural
intuitionistic redex, the new active sessions can ask to transmit only
new messages of smaller value complexity than the complexity of the
fired redex; (ii) after transmitting a message, all the new generated
intuitionistic reductions have at most the value complexity of the
message.  Proposition \ref{lem:replace}) will settle the matter, but
in turn requires a series of preparatory lemmas. Namely, we shall
study how arbitrary substitutions affect the value complexity of a
term in Lemma \ref{lem:change_of_value} and Lemma \ref{lem:replace};
then we shall determine how case reductions impact value complexity in
Lemma \ref{lem:in_case} and Lemma \ref{lem:eliminate_the_case}.

We break the second loop by showing in the crucial
Lemma~\ref{lem:freeze} that message passing, during the communication
phase, cannot produce new active sessions. Intuitively, the new
generated channels and the old duplicated ones are ``frozen'' and only
intuitionistic reductions can activate them, thus with values of
smaller complexity than that of the fired redex.

For clarity, we define here the recursive normalization algorithm that
represents the constructive content of this section's proofs, which
are used to prove the Normalization Theorem.  Essentially, our master
reduction strategy will use in the activation phase the basic
reduction relation $\succ$ defined below, whose goal is to permute an
uppermost active session $\pp{a}{u_1 \p \dots \p u_m}$ until all $u_i$
for $1 \leq i \leq m$ are simply typed $\lambda$-terms and finally
apply the cross reductions followed by projections and case
permutations.


\begin{definition}[Side Reduction Strategy]\label{defi-redstrategy}
Let $t$ be a term and $\pp{a}{u_1 \p \dots \p u_m }$ be an active session
occurring in $t$ such that no active session occurs
in $u$ or $v$. We write\[t\succ
t'\]whenever $t'$ has been obtained from $t$ by applying to
$u\parallel_{a} v$:
\begin{enumerate}
\item a permutation reduction
  \begin{align*}
& \pp{a}{ u_1\p \dots \p \pp{b}{w_1\p \dots \p w_n} \dots \p u_m} \;
\mapsto \\ 
& \pp{b}{ \pp{a}{ u_1\p \dots \p w_1\dots \p u_m}\p \dots \p
\pp{a}{ u_1\p \dots \p w_n\dots \p u_m}}
  \end{align*}
 if  $u_i=\pp{b}{w_1\p \dots \p w_n} $ for some $1 \leq i \leq m$ ;
\item a cross reduction, if $u_1 , \dots , u_m$ are intuitionistic
  terms, immediately followed by the projections and case permutations inside the newly generated simply typed
$\lambda$-terms;
\item a cross reduction $\pp{a}{ u_1 \p  \dots \p u_m } \mapsto u_{j_1} \p \dots \p u_{j_n} $,
for $1 \leq j_1 <  \dots < j_n \leq m $, if $a$ does not occur in
$u_{j_1} , \dots , u_{j_n}$
\end{enumerate}
\end{definition}


\begin{definition}[Master Reduction Strategy]\label{defi-mastredstrategy}
Let $t$ be any term which is not in normal form. We transform it into a term $u$ in parallel form, then we execute the following three-step recursive procedure.
\begin{enumerate}
\item \emph{Intuitionistic Phase}. As long as $u$ contains intuitionistic redexes, we apply intuitionistic reductions.
\item \emph{Activation Phase}. As long as $u$ contains activation redexes, we apply activation reductions.
\item \emph{Communication Phase}. As long as $u$ contains active sessions, we apply the Side Reduction
  Strategy (Definition~\ref{defi-redstrategy}) to $u$, then we go to step $1$.\end{enumerate}
\end{definition}

We start be defining the value complexity of messages. Intuitively, it
is a measure of the complexity of the redexes that a message can generate
after being transmitted. On one hand, it is defined as usual for proper
values, like $s = \lam x u$, $s=\inj_i(u)$, as the complexity of their
types. On the other hand, pairs $\langle u , v \rangle$ and case
distinctions $t[x.u,y.v]$ represent sequences of values, hence we
pick recursively the maximum among the value complexities of $u$ and
$v$. This is a crucial point. If we chose the types as value
complexities also for pairs and case distinctions, then our
argument would completely break down when new channels
are generated during cross reductions: their type can be much bigger
than the starting channel and any shade of a decreasing complexity
measure would disappear.

\begin{definition}[Value Complexity]\label{def:fut_comp} \label{def:value_compl}
For any simply typed $\lambda$-term  $s:T$, the value complexity of $s$ is defined as the first case that applies among the following:
\begin{itemize}
\item if $s = \lam x u$, $s=\inj_i
(u)$, then the value complexity of $s$ is the
complexity of its type $T$;
\item if $s= \langle u , v \rangle $, then the value complexity of $s$ is the
maximum among the value complexities of $u$ and $v$.
\item if $s=t[x.u,y.v]\sigma$ where $\sigma $ is case-free,
then the value complexity of $s$ is the maximum among the value
complexities of $u\sigma$ and $v\sigma$;
\item otherwise, the value
complexity of $s$ is $0$.
\end{itemize}
\end{definition}

Recall that values are defined (Def.~\ref{def:value}) as anything that
either can generate an intuitionistic redex when plugged into another
term or that can be transformed into something with that capability,
like an active channel acting as an endpoint of a transmission.


The value complexity of a term, as expected, is alway at most the complexity of its type. 

\begin{proposition}\label{prop:two!}
Let $u: T$ be any simply typed $\lambda$-term. Then the value complexity of  $u$ is at most the complexity of $T$.
\end{proposition}
\begin{proof}
Induction on the shape of $u$. See Appendix.
\inappendix{
By induction on $u$. There are several cases, according to the shape of $u$.
\begin{itemize}
\item If $u$ is of the form $ \lam x w$, $\inj_i (w)$, then the value complexity of $u$ is indeed the complexity of $T$. 

\item If $u$ is of the form $ \langle v_1 , v_2 \rangle $ then, by
induction hypothesis, the value complexities of $v_{1}$ and $v_{2}$
are at most the complexity of their respective types $T_{1}$ and $T_{2}$, and hence at
most the complexity of $T=T_{1}\et T_{2}$, so we are done.

\item If $u$ is of the form $v_0[z_1. v_{1},z_2.  v_{2}]$ then, by
induction hypothesis, the value complexities of $v_{1}$ and $v_{2}$
are at most the complexity of $T$, so we are done.
\item In all other cases, the value complexity of $u$ is
$0$, which is trivially the thesis. 
\end{itemize}
}
\end{proof}

The complexity of an intuitionistic redex $t \xi$ is defined as the value complexity of $t$.

\begin{definition}[Complexity of the Intuitionistic
Redexes]\label{def:int_comp} Let $r$ be an intuitionistic
redex. The complexity of $r$ is defined as follows:
\begin{itemize}
\item If $r=(\lambda x u)t$, then the complexity of $r$ is the
type of $\lambda x u$.

  \item If $r=\inj_{i}(t)[x.u, y.v]$, then the complexity of $r$ is the
type of $\inj_{i}(t)$.

\item if $r= \lan u,v \ran \pi_i$ then the complexity of $r$ is the
  value complexity of $\lan u,v \ran$.

\item if $r= t [x.u,y.v] \xi $, then the complexity of $r$ is the
  value complexity of $t [x.u,y.v] $. 
\end{itemize}
\end{definition}

The value complexity is used to define the complexity of communication
redexes.  Intuitively, it is the value complexity of the heaviest message ready to be transmitted.

\begin{definition}[Complexity of the Communication Redexes]\label{def:com_comp}
Let $u\parallel_{a} v: A$ a term.
Assume that  $a^{B\rightarrow C}$ occurs in $u$ and $a^{C\rightarrow B}$ in $v$. 
\begin{itemize}
\item The pair $B, C$ is called the \textbf{communication kind}  of $a$. 
\item The \textbf{complexity of a channel occurrence}  $a \, \langle
  t_1, \ldots , t_n \rangle$ is the value complexity of $\langle t_1, \ldots , t_n \rangle$ (see Definition~\ref{def:value_compl}).
\item  The \textbf{complexity of a communication redex} $u\parallel_{a} v$ is the maximum among the complexities of the occurrences of $a$ in $u$ and $v$.

\item The \textbf{complexity of a permutation redex} \\$\pp{a}{ u_1\p
\dots \p \pp{b}{w_1\p \dots \p w_n} \dots \p u_m} $ is $0$.
\end{itemize}
\end{definition}

As our normalization strategy suggests, application and injection redexes should be treated differently from the others, because generate the real computations.

\begin{definition}\label{def:redex_groups}
  We distinguish two groups of redexes:
  \begin{enumerate}
  \item \label{group_one} Group 1: Application and injection redexes.
  \item \label{group_two} Group 2: Communication redexes, projection redexes and
    case permutation redexes.
  \end{enumerate}
\end{definition}

The first step of the normalization proof consists in showing that any
term can be reduced to a parallel form.

\begin{proposition}[Parallel Form]\label{proposition-normpar}
Let $t: A$ be any term. Then $t\mapsto^{*} t'$, where $t'$ is a parallel form. 
\end{proposition}
\begin{proof}
Similar to~\cite{ACGlics2017}.  See Appendix. 
\end{proof}

The following, easy lemma shows that the activation phase of our reduction strategy is finite.

\begin{lemma}[Activate!]\label{activation} 
Let $t$ be any term in parallel form that does not contain intuitionistic redexes and whose communication redexes have  complexity at most $\tau$.  Then there
 exists a finite sequence of activation reductions that results in a
term $t'$ that contains no redexes, except cross reduction redexes of complexity at most $\tau$.
\end{lemma}
\begin{proof} By induction on the number of non-active terms of the
form $\pp{a}{u_1 \p \dots \p u_m}$ in $t$. See Appendix.
\end{proof}

We shall need a simple property of the value complexity notion.

\begin{lemma}[Why Not 0] \label{lem:stacks_and_zeros} Let $u$ be any simply typed $\lambda$-term and $\sigma $
be a non-empty case-free stack. Then
the value complexity of $u\sigma $ is $0$. 
\end{lemma}
\begin{proof}  Induction on the size of $u$. See Appendix.
\inappendix{
By induction on the size of $u$.
  \begin{itemize}
  \item If $u$ is of the form $ (\lam x. w)\rho$, $\inj_i (w)\rho$, $
\langle v_1 , v_2 \rangle\rho $, $a\rho$ or $x\rho $, where $\rho$ is case-free,
then $u\sigma $ has value complexity $0$.

\item If $u$ is of the form $w\, \exfalso_P$, then $P$ is atomic, thus $\sigma$ must be empty,
contrary to the assumptions.

\item If $u$ is of the form $v_0[z_1. v_{1},z_2.  v_{2}]\rho$, with $\rho$ case-free, then
by induction hypothesis the value complexities of
$v_{1}\rho \sigma$ and $v_{2}\rho \sigma$ are $0$ and since the value complexity of
$u\sigma$ is the maximum among them, $u\sigma$ has value complexity $0$.
  \end{itemize}
}
\end{proof}

In order to formally study redex contraction, we
consider simple substitutions that just replace some occurrences of a
term with another, allowing capture of variables. In practice, it will
always be clear from the context which occurrences are replaced.

\begin{definition}[Simple Replacement] By $s\{t / u \}$
we denote a term obtained from $s$ be replacing some occurrences of a term $u$ with a term $t$ of the same type of $u$, possibly causing capture
of variables.
\end{definition}

We now show an important property of value complexity:  the value complexity of $ w \{v/s\} $ either remains at most as it was before the substitution or becomes exactly the value complexity of $v$.

\begin{lemma}[The Change of Value]\label{lem:change_of_value}

Let $w, s,v$ be simply typed $\lambda$-terms with value complexity
respectively $\theta , \tau, \tau'$.  Then the value complexity of
$ w \{v/s\} $ is either at most $\theta$ or equal to  $\tau'$. Moreover, if $\tau '
\leq \tau $, then the value complexity of $ w \{v/s\} $ is at most $\theta
$.
\end{lemma}
\begin{proof} Induction on the size of $w$. See Appendix.
\end{proof}

The following lemma has the aim of studying mostly message passing and contraction of application and projection redexes.

\begin{lemma}[Replace!]\label{lem:replace} 
Let $u$ be a term in parallel form, $v$, $s$ be any simply typed $\lambda$-terms,  $\tau$ be the
value complexity of $v$ and $\tau'$ be the maximum among the complexities of the channel occurrences in $v$. Then every  redex in $u\{v/s\}$ it is either (i) already in
$v$, (ii) of the form $r\{v/s\}$ and has complexity smaller than or
equal to the complexity of some redex $r$ of $u$, or (iii) has complexity  $\tau$   or is a communication redex of complexity at most $\tau'$.
\end{lemma}
\begin{proof}
Induction on the size of $u$. See Appendix.
\end{proof}

Below we study the complexity of  redexes generated after contracting an injection redex.

\begin{lemma}[Eliminate the Case!]\label{lem:eliminate_the_case} Let $u$ be a
term in parallel form.  Then for any redex $r$ in $u\{ w_i [t / x_i] /
\inj _i (t)[x_1.w_1, x_2.w_2] \}$ of complexity $\theta$, either $\inj _i (t)[x_1.w_1, x_2.w_2]$ has complexity greater than $\theta$; or
 there is a
redex in $u$ of complexity $\theta$ which belongs to the same group as $r$ or is a case permutation redex.
\end{lemma}
\begin{proof} Induction on the size of $u$. See Appendix.
\end{proof}

We analyze  what happens after permuting a case distinction.

\begin{lemma}[In Case!]\label{lem:in_case} Let $u$ be a term in
parallel form. Then for any redex $r_1$ of Group~\ref{group_one} in
$u\{ t[x_1.v_1\xi, x_2.v_2\xi] / t[x_1.v_1,x_2.v_2]\xi \}$, there is a
redex in $u$ with greater or equal complexity than $r_1$; for any
redex $r_2$ of Group~\ref{group_two} in $u\{ t[x_1.v_1\xi, x_2.v_2\xi]
/ t[x_1.v_1,x_2.v_2]\xi \}$, there is a redex of Group~\ref{group_two}
in $u$ with greater or equal complexity than $r_2$.
\end{lemma}
\begin{proof}
  Induction on the shape of $u$. See Appendix.
\end{proof}

The following result is meant to break the possible loop between the intuitionistic phase and communication phase of our normalization strategy. Intuitively, when Group~\ref{group_one} redexes generate new redexes, these latter have smaller complexity than the former; when Group~\ref{group_two} redexes generate new redexes, these latter does not have worse complexity than the former. 

\begin{proposition}[Decrease!]\label{prop:decrease} Let $t$ be a term
in parallel form, $r$ be one of its redexes of complexity $\tau$, and
$t'$ be the term that we obtain from $t$ by contracting $r$.

$1.$ If $r$ is a redex of the Group~\ref{group_one}, then the
complexity of each redex in $t'$ is at most the complexity of a redex
of the same group and occurring in $t$; or is at most the complexity
of a case permutation redex occurring in $t$; or is smaller than
$\tau$.

$2.$ If $r$ is a redex of the Group~\ref{group_two} and not an
activation redex, then every redex in $t'$ either has the complexity
of a redex of the same group occurring in $t$ or has complexity at
most $\tau$.
\end{proposition}

\begin{proof} See Appendix.
\end{proof}

The following result is meant to break the possible loop during the communication phase: no new activation is generated after a cross reduction, when there is none to start with.



\begin{lemma}[Freeze!]\label{lem:freeze} Suppose that $s$ is a term in
parallel form that does not contain projection nor case permutation
nor activation redexes.  Let $\pp{a}{q_1 \p \dots \p q_m }$
be some redex of $s$ of
complexity $\tau$. If $s'$ is obtained from $s$ by performing first a
cross reduction on $\pp{a}{q_1 \p \dots \p q_m }$
and then contracting all projection
and case permutation redexes, then $s'$ contains no activation
redexes. \end{lemma}
 \begin{proof} We reason on the simple replacements produced by cross
   reductions, case permutations and 
projections. See Appendix.
\end{proof}

\begin{definition} The height $\mathsf{h}(t)$ of a term
$t$ in parallel form is 
    \begin{itemize}
    \item $\mathsf{h}( u ) \; = \; 0$ if $u$ is simply typed
$\lambda$-term
    \item $\mathsf{h}( u \parallel _a v ) \; = \; 1+ \,
\text{max}(\mathsf{h}( u ) , \mathsf{h} ( v ) )$
\end{itemize}
  \end{definition}

The communication phase of our reduction strategy is finite.

\begin{lemma}[Communicate!]\label{lem:com_phase} Let $t$ be any term
in parallel form that does not contain projection, case permutation,
or activation redexes.
Assume moreover that all  redexes in $t$ have complexity at most $\tau$. 
Then $t$ reduces to a term
containing no redexes, except Group~\ref{group_one} redexes of complexity at most $\tau$.
\end{lemma}
\begin{proof} We reason by lexicographic induction on a triple defined
using the number of non-uppermost active sessions in $t$, and the height and
number of channels of the uppermost active sessions in $t$. See Appendix.
\end{proof}

We now combine together all the main results achieved so far.

\begin{proposition}[Normalize!]\label{prop:normalize} Let $t: A$ be any term
in parallel form. Then $t \mapsto^{*} t'$, where $t'$ is a
parallel normal form.
\end{proposition}
\begin{proof} 
{
Let $\tau$ be the maximum among the complexity of the redexes in $t$. We prove
the statement by induction on $\tau$. Starting from $t$, we reduce all intuitionistic redexes and obtain a
term $t_1$ that, by Proposition~\ref{prop:decrease}, does not contain
redexes of complexity greater than $\tau$. By Lemma~\ref{activation},
$t_1\mapsto ^* t_2$ where $t_2$ does not contain any redex, except
cross reduction redexes of complexity at most $\tau$.  By
Lemma~\ref{lem:com_phase}, $t_2\mapsto^* t_3$ where $t_3$ contains
only Group~\ref{group_one} redexes of complexity at most $\tau$. Suppose
$t_3\mapsto^* t_4$ by reducing all Group~\ref{group_one} redexes, starting from
$t_{3}$. By Proposition~\ref{prop:decrease}, every Group~\ref{group_one} redex
 generated in the process has complexity at most $\tau$, thus
every Group~\ref{group_two} redex which is generated has complexity smaller than
$\tau$ , thus $t_{4}$ can only contain redexes with complexity smaller
than $\tau$. By induction hypothesis $t_4 \mapsto ^* t'$, with $t'$ in
parallel normal form.
}
\end{proof}

  The normalization for $\lama$, now easily follows.

\begin{theorem}\label{theorem-normalization} Suppose that  $ t: A$ is
  a $\lama$ proof-term. Then $t\mapsto^{*} t': A$, where $t'$ is a normal parallel form.
\end{theorem}

\begin{remark}\label{rem:cond} This normalization proof covers all systems
corresponding to axioms in the class~(\ref{ouraxiom}) but the 
the notation of typing rules and reductions becomes much more
involved.
\end{remark}

\section{ 
The Subformula Property}\label{section-subformula}

We show that normal $\lama$-terms satisfy the
Subformula Property:  a normal proof does not
contain concepts that do not already appear in the premisses or in the
conclusion. This, in turn,
implies that our Curry--Howard correspondence for $\lama$ is meaningful
from the logical perspective and produces analytic  proofs.

\begin{proposition}[Parallel Normal Form Property]
\label{proposition-parallelform} 
If $t\in
\nf$ is a $\lama$-term, then it is in parallel form.
\end{proposition}
\begin{proof}
Easy induction on the shape of $t$. 
\end{proof}

\begin{theorem}[Subformula Property]\label{theorem-subformula} Suppose
\[x_{1}^{A_{1}}, \ldots, x_{n}^{A_{n}}, a_{1}^{D_{1}}, \ldots, a_{m}^{D_{m}}\vdash t: A \quad \mbox{and} \quad
t\in \nf. \quad \mbox{Then}:\]
\begin{enumerate}
\item \label{dot:subf_1}
For each channel variable $a^{B \IMPL C}$ occurring bound in
$t$, the prime factors of $B,C$ are
subformulas of $A_{1}, \ldots, A_{n}, A$ or proper subformulas of $D_{1}, \ldots, D_{m}$.
\item \label{dot:subf_2}
 The type of any subterm of $t$ 
 is either a subformula or a conjunction of subformulas of
$A_{1}, \ldots, A_{n}, A$ and of proper subformulas of $D_{1}, \ldots, D_{m}$.
\end{enumerate}
\end{theorem}
\begin{proof} 
Induction on the shape of $t$. See Appendix.
\end{proof}

\newpage

\appendix

\section{The Normalization Theorem}

Our goal is to prove the Normalization Theorem for $\lama$:
every proof term of $\lama$ reduces in a finite number of steps to a
normal form.  By Subject Reduction, this implies that the corresponding natural deduction proofs do
normalize. We shall define a reduction strategy for terms of $\lama$:
a recipe for selecting, in any given term, the subterm to which apply
one of our basic reductions. As opposed to \cite{ACGlics2017},  our permutations between communications do not enable silly loops and thus do not undermine the possibility that our set of reduction is strongly normalizing. We leave anyway as a difficult, open problem to determine whether this is really the case.

The idea behind our normalization strategy is quite intuitive. We start from any term and reduce it in parallel normal form, thanks to Proposition \ref{proposition-normpar}. Then we  cyclically interleave three reduction phases. First, an \emph{intuitionistic phase}, where we reduce all intuitionistic redexes. Second, an \emph{activation phase}, where we activate all sessions that can be activated. Third, a \emph{communication phase}, where we allow the active sessions to exchange messages as long as they need and we enable the receiving process to extract information from the messages. Technically, we perform all cross reductions combined with the generated structural intuitionistic redexes, which we consider to be projections and case permutations.

Proving termination of this strategy is by no means easy, as we have to rule out two possible infinite loops.

\begin{enumerate}
\item Intuitionistic reductions can generate new activable sessions that want to transmit messages, while message exchanges can generate new intuitionistic reductions.

\item During the communication phase, new sessions may be generated after each cross reduction and old sessions may be duplicated after each session permutation. The trouble is that each of these sessions may need to send new messages, for instance forwarding some message received from some other active session. Thus the count of active sessions might increase forever and the communication phase never terminate.
\end{enumerate}

We break the first loop by focusing on the complexity of the exchanged messages. Since messages are \emph{values}, we shall define a notion of \emph{value complexity} (Definition \ref{def:fut_comp}), which will simultaneously ensure that: (i) after firing a non-structural intuitionistic redex, the new active sessions can ask to transmit only new messages of smaller value complexity than the complexity of the fired redex; (ii) after transmitting a message, all the new generated intuitionistic reductions have at most the value complexity of the message.  Proposition \ref{lem:replace}) will settle the matter, but in turn requires a series of preparatory lemmas. Namely, we shall study how arbitrary substitutions affect the value complexity of a term in Lemma \ref{lem:change_of_value} and Lemma \ref{lem:replace}; then we shall determine how case reductions impact value complexity  in Lemma \ref{lem:in_case} and Lemma \ref{lem:eliminate_the_case}. 

We break the second loop by showing in the crucial Lemma
\ref{lem:freeze} that message passing, during the communication phase,
cannot produce new active sessions. Intuitively, the new generated
channels or the old duplicated ones are ``frozen'' and only
intuitionistic reductions can activate them, thus with values of
smaller complexity than that of the fired redex.

For clarity, we define here the recursive normalization algorithm that
represents the constructive content of this section's proofs, which
are used to prove the Normalization Theorem.  Essentially, our master
reduction strategy will use in the activation phase the basic
reduction relation $\succ$ defined below, whose goal is to permute an
uppermost active session $\pp{a}{u_1 \p \dots \p u_m}$ until all $u_i$
for $1 \leq i \leq m$ are simply typed $\lambda$-terms and finally
apply the cross reductions followed by projections and case
permutations.


\begin{definition}[Side Reduction Strategy]
Let $t$ be a term and $\pp{a}{u_1 \p \dots \p u_m }$ be an active session
occurring in $t$ such that no active session occurs
in $u$ or $v$. We write\[t\succ
t'\]whenever $t'$ has been obtained from $t$ by applying to
$u\parallel_{a} v$:
\begin{enumerate}
\item a permutation reduction
  \begin{align*}
& \pp{a}{ u_1\p \dots \p \pp{b}{w_1\p \dots \p w_n} \dots \p u_m} \;
\mapsto \\ 
& \pp{b}{ \pp{a}{ u_1\p \dots \p w_1\dots \p u_m} \dots \p
\pp{a}{ u_1\p \dots \p w_1\dots \p u_m}}
  \end{align*}
 if  $u_i=\pp{b}{w_1\p \dots \p w_n} $ for some $1 \leq i \leq m$ ;
\item a cross reduction, if $u_1 , \dots , u_m$ are intuitionistic
  terms, immediately followed by the projections and case permutations inside the newly generated simply typed
$\lambda$-terms;
\item a cross reduction $\pp{a}{ u_1 \p  \dots \p u_m } \mapsto u_{j_1} \p \dots \p u_{j_n} $,
for $1 \leq j_1 <  \dots < j_n \leq m $, if $a$ does not occur in
$u_{j_1} , \dots , u_{j_n}$
\end{enumerate}
\end{definition}


\begin{definition}[Master Reduction Strategy]
Let $t$ be any term which is not in normal form. We transform it into a term $u$ in parallel form, then we execute the following three-step recursive procedure.
\begin{enumerate}
\item \emph{Intuitionistic Phase}. As long as $u$ contains intuitionistic redexes, we apply intuitionistic reductions.
\item \emph{Activation Phase}. As long as $u$ contains activation redexes, we apply activation reductions.
\item \emph{Communication Phase}. As long as $u$ contains active sessions, we apply the Side Reduction
  Strategy (Definition~\ref{defi-redstrategy}) to $u$, then we go to step $1$.\end{enumerate}
\end{definition}

We start be defining the value complexity of messages. Intuitively, it is a measure of how much complex redexes a message can generate, after being transmitted. On one hand, it is defined as usual for true values, like $s = \lam x u$, $s=\inj_i(u)$, as the complexity of their types. On the other hand, pairs $\langle u , v \rangle$ and case distinctions $t[x.u,y.v]$ represents sequences of values, therefore we pick recursively the maximum among the value complexities of $u$ and $v$. This is a crucial point.  If we chose the types as value complexities also for pairs and case distinctions, then our global normalization argument would  completely break down when new channels are generated during cross reductions: their type can be much bigger than the starting channel and any shade of a decreasing complexity measure would disappear.

\begin{definition}[Value Complexity]
For any simply typed $\lambda$-term  $s:T$, the value complexity of $s$ is defined as the first case that applies among the following:
\begin{itemize}
\item if $s = \lam x u$, $s=\inj_i
(u)$, then the value complexity of $s$ is the
complexity of its type $T$;
\item if $s= \langle u , v \rangle $, then the value complexity of $s$ is the
maximum among the value complexities of $u$ and $v$.
\item if $s=t[x.u,y.v]\sigma$ where $\sigma $ is case-free,
then the value complexity of $s$ is the maximum among the value
complexities of $u\sigma$ and $v\sigma$;
\item otherwise, the value
complexity of $s$ is $0$.
\end{itemize}
\end{definition}

Values are defined as anything that either can generate an intuitionistic redex when plugged into another term or that can be transformed into something with that capability, like an active channel acting as an endpoint of a transmission. 

\begin{definition}[Value]
A \textbf{value} is a term of the form $ \langle t_1, \ldots , t_n
\rangle $, for some $1 \leq i \leq n $, $t_i =
\lambda x \, s$, $t_i = \inj _i (s)$, $t_i = t\, \exfalso_P $, $t_{i}=t
[x.u,y.v]$ or $t_{i}=a\sigma$ for an
active channel $a$.  
\end{definition}

The value complexity of a term, as expected, is alway at most the complexity of its type. 

\begin{proposition}
Let $u: T$ be any simply typed $\lambda$-term. Then the value complexity of  $u$ is at most the complexity of $T$.
\end{proposition}
\begin{proof}
{
By induction on $u$. There are several cases, according to the shape of $u$.
\begin{itemize}
\item If $u$ is of the form $ \lam x w$, $\inj_i (w)$, then the value complexity of $u$ is indeed the complexity of $T$. 

\item If $u$ is of the form $ \langle v_1 , v_2 \rangle $ then, by
induction hypothesis, the value complexities of $v_{1}$ and $v_{2}$
are at most the complexity of their respective types $T_{1}$ and $T_{2}$, and hence at
most the complexity of $T=T_{1}\et T_{2}$, so we are done.

\item If $u$ is of the form $v_0[z_1. v_{1},z_2.  v_{2}]$ then, by
induction hypothesis, the value complexities of $v_{1}$ and $v_{2}$
are at most the complexity of $T$, so we are done.
\item In all other cases, the value complexity of $u$ is
$0$, which is trivially the thesis. 
\end{itemize}
}
\end{proof}

The complexity of an intuitionistic redex $t \xi$ is defined as the value complexity of $t$.

\begin{definition}[Complexity of the Intuitionistic
Redexes]
Let $r$ be an intuitionistic
redex. The complexity of $r$ is defined as follows:
\begin{itemize}
\item If $r=(\lambda x u)t$, then the complexity of $r$ is the
type of $\lambda x u$.

  \item If $r=\inj_{i}(t)[x.u, y.v]$, then the complexity of $r$ is the
type of $\inj_{i}(t)$.

\item if $r= \lan u,v \ran \pi_i$ then the complexity of $r$ is the
  value complexity of $\lan u,v \ran$.

\item if $r= t [x.u,y.v] \xi $, then the complexity of $r$ is the
  value complexity of $t [x.u,y.v] $. 
\end{itemize}
\end{definition}

The value complexity is used to define the complexity of communication redexes. Intuitively, it is the value complexity of the heaviest message ready to be transmitted.

\begin{definition}[Complexity of the Communication Redexes]
Let $u\parallel_{a} v: A$ a term.
Assume that  $a^{B\rightarrow C}$ occurs in $u$ and thus $a^{C\rightarrow B}$ in $v$. 
\begin{itemize}
\item The pair $B, C$ is called the \textbf{communication kind}  of $a$. 
\item The \textbf{complexity of a channel occurrence}  $a \, \langle
  t_1, \ldots , t_n \rangle$ is the value complexity of $\langle t_1, \ldots , t_n \rangle$ (see Definition~\ref{def:value_compl}).
\item  The \textbf{complexity of a communication redex} $u\parallel_{a} v$ is the maximum among the complexities of the occurrences of $a$ in $u$ and $v$.

\item The \textbf{complexity of a permutation redex} \\$\pp{a}{ u_1\p
\dots \p \pp{b}{w_1\p \dots \p w_n} \dots \p u_m} $ is $0$.
\end{itemize}
\end{definition}

As our normalization strategy suggests, application and injection redexes should be treated differently from the others, because generate the real computations.

\begin{definition}
  We distinguish two groups of redexes:
  \begin{enumerate}
  \item 
    Group 1: Application and injection redexes.
  \item 
    Group 2: Communication redexes, projection redexes and
    case permutation redexes.
  \end{enumerate}
\end{definition}

The first step of the normalization proof consists in showing that any
term can be reduced to a parallel form.

\begin{proposition}[Parallel Form]
Let $t: A$ be any term. Then $t\mapsto^{*} t'$, where $t'$ is a parallel form. 
\end{proposition}
\begin{proof}
{
By induction on $t$. As a shortcut, if a term $u$ reduces to a
term $u'$ that can be denoted as $u''$ omitting parentheses, we
write $u \mapstopar^{*} u''$. 
\begin{itemize}
\item  $t$ is a variable $x$. Trivial. 
\item $t=\lambda x\, u$. By induction hypothesis,  
$u\mapstopar^{*} u_{1}\parallel u_{2}\parallel \ldots \parallel u_{n+1}$
and each term $u_{i}$, for $1\leq i\leq n+1$, is a simply typed
$\lambda$-term.
Applying several  
permutations we obtain \[t\mapstopar^{*} \lambda x\, u_{1}\parallel \lambda x\, u_{2}\parallel
  \ldots \parallel \lambda x\, u_{n+1}\]
which is the thesis.


\item $t=u\, v$. 
By induction hypothesis, 
\[u\mapstopar^{*} u_{1}\parallel u_{2}\parallel \ldots \parallel u_{n+1}\]
\[v\mapstopar^{*} v_{1}\parallel v_{2}\parallel \ldots \parallel v_{m+1}\]
and each term $u_{i}$ and $v_{i}$, for $1\leq i\leq n+1, m+1$, is a
simply typed $\lambda$-term. Applying several
permutations we obtain
\[
\begin{aligned}
t &\mapstopar^{*} (u_{1}\parallel u_{2}\parallel \ldots \parallel  u_{n+1})\, v \\
&\mapstopar^{*}  u_{1}\, v \parallel  u_{2}\, v \parallel
\ldots \parallel  u_{n+1}\, v\\
&\mapstopar^{*} u_{1}\, v_{1} \parallel u_{1}\, v_{2}\parallel
\ldots \parallel u_{1}\, v_{m+1} \parallel \ldots
\\
& \qquad  \, \ldots \parallel u_{n+1}\, v_{1} \parallel  u_{n+1}\,
v_{2} \parallel \ldots
 \parallel  u_{n+1}\, v_{m+1}
\end{aligned}
\]

\item $t=\langle u, v\rangle$. By induction hypothesis, 
$$u\mapstopar^{*} u_{1}\parallel  u_{2}\parallel \ldots \parallel u_{n+1}$$
$$v\mapstopar^{*} v_{1}\parallel v_{2}\parallel \ldots \parallel v_{m+1}$$
and each term $u_{i}$ and $v_{i}$, for $1\leq i\leq n+1, m+1$, is a
simply typed $\lambda$-term. Applying several 
 permutations we
obtain
  \[
    \begin{aligned}
      t &\mapstopar^{*}\langle u_{1}\parallel u_{2}\parallel
      \ldots \parallel  u_{n+1},\, v \rangle\\
      &\mapstopar^{*} \langle u_{1}, v\rangle \parallel \langle u_{2}, v
      \rangle\parallel \ldots \parallel \langle u_{n+1}, v\rangle\\
      &\mapstopar^{*} \langle u_{1}, v_{1}\rangle \parallel
      \langle u_{1}, v_{2}
      \rangle\parallel \ldots \parallel \langle u_{1},
      v_{m+1}\rangle \parallel  \ldots
      \\
      & \qquad \, \ldots
      \parallel \langle u_{n+1},
      v_{1}\rangle \parallel \langle u_{n+1}, v_{2}
      \rangle\parallel \ldots
      \\
      & \qquad \, \ldots \parallel \langle u_{n+1},
      v_{m+1}\rangle
    \end{aligned}
  \]

\item $t=u\, \pi_{i}$. By induction hypothesis,
$$u\mapstopar^{*} u_{1}\parallel  u_{2}\parallel \ldots \parallel u_{n+1}$$
and each term $u_{i}$, for $1\leq i\leq n+1$, is a simply typed
$\lambda$-term. Applying several 
 permutations we obtain
$$t\mapstopar^{*}  u_{1}\, \pi_{i}\parallel  u_{2}\,
\pi_{i}\parallel \ldots \parallel u_{n+1} \, \pi_{i}.$$

\item $t= \efq{P}{u} $. By induction hypothesis,
$$u\mapstopar^{*} u_{1}\parallel  u_{2}\parallel \ldots \parallel u_{n+1}$$
and each term $u_{i}$, for $1\leq i\leq n+1$, is a simply typed
$\lambda$-term. Applying several
 permutations we obtain
\[t\mapstopar^{*} \efq{P}{u_{1}} \parallel \efq{P}{u _2 } \parallel
\ldots \parallel \efq{P}{u_{n+1} }\]
\end{itemize}
}
\end{proof}

The following, easy lemma shows that the activation phase of our reduction strategy is finite.

\begin{lemma}[Activate!]
  
Let $t$ be any term in parallel form that does not contain intuitionistic redexes and whose communication redexes have  complexity at most $\tau$.  Then there
 exists a finite sequence of activation reductions that results in a
term $t'$ that contains no redexes, except cross reduction redexes of complexity at most $\tau$.
\end{lemma}
\begin{proof} 
{
The proof is by induction on the number $n$ of subterms
of the form $\pp{a}{u_1 \p \dots \p u_m}$ of $t$ which are not active sessions. If there are no activation redexes in $t$, the statement
trivially holds. Assume there is
at least one activation redex  $r=  \pp{a}{v_1 \p \dots \p v_m}$. 
We apply an activation reduction to
$r$ and obtain a term $r'$ with $n-1$ subterms of the form  
$\pp{a}{u_1 \p \dots \p u_m}$ which are not active sessions. By the
induction hypothesis to $r'$, which immediately yields the thesis, we
are left to verify that all communication redexes of $r'$ have
complexity at most $\tau$.

\noindent  For this purpose, let $c$ be any channel variable which is bound in $r'$.  Since $r'$ is obtained from $r$ just by renaming the non-active bound channel variable $a$ to an active one $\alpha$, every occurrence of $c$ in $r'$ is of the form $(c t)[\alpha/a]$ for some subterm $ct$ of $r$. Thus  $ct[\alpha/a] =c[\alpha/a] \langle t_1[\alpha/a], \dots, t_n[\alpha/a] \rangle $, where each $t_{i}$ is not a pair. It is enough to  show that the value complexity of $t_i [\alpha/a] $ is exactly the value complexity of $t_i$. We proceed by induction on the size of $t_{i}$. We can write   $t_{i}= r \, \sigma $, where $\sigma$ is
  a case-free  stack. If $r$ is of the form $\lam x w $, $\lan q_1 , q_2 \ran $,
$\inj _i (w)$, $x$, $d w$, with $d$ channel variable, then the value complexity of $t_i [\alpha/a] $ is the same as that of $t_{i}$ (note that if $r=\lan q_{1}, q_{2}\ran$, then $\sigma$ is not empty). If $r= v_{0}[x_{1}.v_{1}, x_{2}.v_{2}]$, then $\sigma$ is empty, otherwise $s$ would contain a permutation redex. Therefore, the value complexity of $t_{i}[\alpha/a]$ is the maximum among the value complexities of $v_{1}[\alpha/a]$ and $v_{2}[\alpha/a]$. By induction hypothesis, their value complexities are respectively those of $v_{1}$ and $v_{2}$ , hence the value complexity of $t_{i}[\alpha/a]$ is the same as that of $t_{i}$, which concludes the proof.
}
\end{proof}

We shall need a simple property of the value complexity notion.

\begin{lemma}[Why Not 0] 
  Let $u$ be any simply typed $\lambda$-term and $\sigma $
be a non-empty case-free stack. Then
the value complexity of $u\sigma $ is $0$. 
\end{lemma}
\begin{proof}  
{
By induction on the size of $u$.
  \begin{itemize}
  \item If $u$ is of the form $ (\lam x. w)\rho$, $\inj_i (w)\rho$, $
\langle v_1 , v_2 \rangle\rho $, $a\rho$ or $x\rho $, where $\rho$ is case-free,
then $u\sigma $ has value complexity $0$.

\item If $u$ is of the form $w\, \exfalso_P$, then $P$ is atomic, thus $\sigma$ must be empty,
contrary to the assumptions.

\item If $u$ is of the form $v_0[z_1. v_{1},z_2.  v_{2}]\rho$, with $\rho$ case-free, then
by induction hypothesis the value complexities of
$v_{1}\rho \sigma$ and $v_{2}\rho \sigma$ are $0$ and since the value complexity of
$u\sigma$ is the maximum among them, $u\sigma$ has value complexity $0$.
  \end{itemize}
}
\end{proof}

As usual, in order to formally study redex contraction, we must consider simple substitutions that just replace some occurrences of a term with another, allowing capture of variables. In practice, it will always be clear from the context which of these occurrences will be replaced.

\begin{definition}[Simple Replacement] By $s\{t / u \}$
we denote a term obtained from $s$ be replacing some occurrences of a term $u$ with a term $t$ of the same type of $u$, possibly causing capture
of variables.
\end{definition}

We now show an important property of the value complexity notion:  the value complexity of $ w \{v/s\} $ either remains at most as it was before the substitution or becomes exactly the value complexity of $v$.

\begin{lemma}[The Change of Value]

Let $w, s,v$ be simply typed $\lambda$-terms with value complexity
respectively $\theta , \tau, \tau'$.  Then the value complexity of
$ w \{v/s\} $ is either at most $\theta$ or equal to  $\tau'$. Moreover, if $\tau '
\leq \tau $, then the value complexity of $ w \{v/s\} $ is at most $\theta
$.
\end{lemma}
\begin{proof}
{
 By induction on the size of $w$ and by cases according to its possible shapes.
\begin{itemize}
\item $w  \{v/s\} = x  \, \{v/s\}$. We have two cases.
  \begin{itemize}
  \item $s = x$. Then the value complexity of $w \{v/s\}=v$ is $\tau'$. Moreover, if
$\tau ' \leq \tau $, since $w=x=s$, we have $\theta = \tau$, thus $\tau'\leq \theta$. 

  \item $s \neq x$.  The value complexity of $w \{v/s\}$ is $\theta$
and we are done.
  \end{itemize}

\item  $w  \{v/s\} = \lam x u \, \{v/s\}$. We have two cases.
  \begin{itemize}
  \item $\lam x u \, \{v/s\} = \lam x (u \{v/s\})$. Then the
    value
complexity of $w \{v/s\}$ is $\theta$ and we are done.
\item $\lam x u \, \{v/s\} = v$. Then the value complexity of $w \{v/s\}$
is $\tau'$. Moreover, if $\tau ' \leq \tau $, since $w= \lam x u =s$, we have
$\theta = \tau$, thus $\tau'\leq \theta$. 
\end{itemize}

\item $ w \{v/s\} = \lan q_1 , q_2 \ran \, \{v/s\} $. We have two cases.

  \begin{itemize}
  \item $\langle q_1 , q_2 \rangle \, \{v/s\} = \langle q_1 \{v/s\} ,
q_2 \{v/s\} \rangle$.  Then by induction hypothesis, the value
complexities of $q_1 \{v/s\}$ and $q_2 \{v/s\}$
are at most $\theta$ or equal to $\tau '$.
Since  the value complexity of $w \{v/s\}$ is the maximum among the
value complexities of $q_1 \{v/s\}$ and $q_2 \{v/s\}$, we are done. Moreover, if $\tau ' \leq \tau $, by induction hypothesis, the value
complexities of $q_1 \{v/s\}$ and $q_2 \{v/s\}$ are at most $\theta$. Hence the value complexity of $w \{v/s\}$ is
at most  $\theta$ and we are done again.
  \item $\langle q_1 , q_2 \rangle \, \{v/s\} = v$.  Then the value
complexity of $w \{v/s\}$ is $\tau'$.  Moreover, if $\tau ' \leq \tau $, since
$w= \langle q_1 , q_2 \rangle =s$, we have $\theta = \tau$, thus $\tau'\leq \theta$. 
\end{itemize}

\item $w \{v/s\} = \inj_i (u) \, \{v/s\}$. We have two cases.
  \begin{itemize}
  \item $\inj_i (u) \, \{v/s\} = \inj_i (u\{v/s\})$. Then the value
complexity of $w \{v/s\}$ is $\theta$ and we are done.
  \item $\inj_i (u) \, \{v/s\} = v$. Then the value complexity of $w
\{v/s\}$ is $\tau'$. Moreover, if $\tau ' \leq \tau $, since
$w= \inj_i (u) =s$,  we have $\theta = \tau$, thus $\tau'\leq \theta$. 
\end{itemize}

\item
\begin{itemize} 
\item $w\{v/s\}=(v_0[z_1.  v_{1}, z_2. v_{2}] ) \rho \{v/s\} =$\\$ v_0 \{v/s\}
[z_1. v_{1}\{v/s\},z_2.  v_{2}\{v/s\}] (\rho \{v/s\})$, where $\rho$ is a case-free stack.  If $\rho$ is not empty, then by Lemma \ref{lem:stacks_and_zeros}, $v_{1}\{v/s\}\rho \{v/s\}$ and $v_{2}\{v/s\}\rho \{v/s\}$ have value complexity $0$, so $w\{v/s\}$ has value complexity $0\leq \theta$ and we are done. So assume  $\rho$ is empty. By
induction hypothesis applied to $v_{1}\{v/s\}$ and
$v_2 \{v/s\}$, the value complexity of $w \{v/s\} $ is at most
$\theta$ or equal to $\tau'$ and we are done. Moreover, if $\tau '
\leq \tau $, then by inductive hypothesis, the value
complexity of $ v_i \{v/s\} $ for $i \in \{1,2\}$ is at most
$\theta $. Hence, the value complexity of \\$ (v_0[z_1. v_{1},z_2.
v_{2}] )\rho \{v/s\}$ is at most $\theta$ and we are done.

\item  $w\{v/s\}=(v_0[z_1.
v_{1},z_2. v_{2}]\rho)
 \{v/s\} = v \rho_j \{v/s\} \dots \rho_n\{v/s\}$, where $\rho = \rho_1 \dots \rho_n$ is a case-free stack and $1 \leq
j \leq n$. If $\rho_j  \dots \rho_n$ is not empty, then by Lemma \ref{lem:stacks_and_zeros}, the value complexity of $w\{v/s\}$ is $0\leq \theta$, and we are done.  So assume $\rho_j  \dots \rho_n$ is empty.  Then the value complexity of $w\{v/s\}$ is $\tau'$. Moreover, if $\tau'\leq \tau$, since it must be $s=v_0[z_1.
v_{1},z_2. v_{2}]$, we have that the value complexity of $w\{v/s\}=v$ is $\tau'\leq \tau=\theta$.

\end{itemize}

\item In all other cases,  $ w \{v/s\} = (r \, \rho ) \{v/s\}$ where $\rho$ is
  a case-free non-empty stack and $r$ is of the form $\lam x u $, $\lan q_1 , q_2 \ran $,
$\inj _i (u)$, $x$, or $au$. We distinguish three  cases.
  \begin{itemize}
  \item $( r \, \rho) \{v/s\} = r\{v/s\} \rho  \{v/s\}$ and $r\neq
s$. Then the value complexity of $w \{v/s\}$ is $0 \leq \theta $ and we are done.
  \item  $(r \, \rho ) \{v/s\} = v
\rho_j\{v/s\} \dots \rho _n\{v/s\}  $, with $\rho =
\rho_1 \dots \rho _n$ and $1 \leq j \leq  n $.  Then  by Lemma \ref{lem:stacks_and_zeros} the value
complexity of $w \{v/s\}$ is $0 \leq \theta $ and we are done.

  \item $r \rho \{v/s\} = v $.  Then the value complexity of $w
\{v/s\}$ is $\tau '$. Moreover, if $\tau ' \leq \tau $, since
$w= r \rho  =s$, we have $\theta = \tau$, thus $\tau'\leq \theta$. 
  \end{itemize}
\end{itemize}
}
\end{proof}

The following lemma has the aim of studying mostly message passing and contraction of application and projection redexes.

\begin{lemma}[Replace!]
  
Let $u$ be a term in parallel form, $v$, $s$ be any simply typed $\lambda$-terms,  $\tau$ be the
value complexity of $v$ and $\tau'$ be the maximum among the complexities of the channel occurrences in $v$. Then every  redex in $u\{v/s\}$ it is either (i) already in
$v$, (ii) of the form $r\{v/s\}$ and has complexity smaller than or
equal to the complexity of some redex $r$ of $u$, or (iii) has complexity  $\tau$   or is a communication redex of complexity at most $\tau'$.
\end{lemma}
\begin{proof}
{
We prove the following stronger statement.
\smallskip 

\noindent $(*)$ Every  redex and channel occurrence in $u\{v/s\}$ it is either (i) already in
$v$, (ii) of the form
$r\{v/s\}$ or $aw\{v/s\}$ and has complexity smaller than or
equal to the complexity of some redex $r$ or channel occurrence $aw$  of $u$, or (iii) has complexity
$\tau$ or $\tau'$ or is a communication redex of complexity at most $\tau'$.
\smallskip

We reason by induction on the size and by cases on the possible  shapes of the term $u$. 
\begin{itemize}

\item $(\lambda x \, t )\, \sigma\, \{v/s\}$, where $\sigma=
\sigma_{1}\ldots \sigma_{n}$ is any case-free stack. By induction hypothesis,
$(*)$ holds for $t \{v/s\}$ and $\sigma_i\{v/s\}$ where $1\leq i \leq
n$. If $ (\lambda x \, t )\, \sigma \{v/s\} = (\lambda x \,
t\{v/s\})\, (\sigma \{v/s\})$, all the redexes and channel occurrences
that we have to check are either in $t \{v/s\}, \sigma \{v/s\}$ or 
possibly, the head redex, thus the thesis holds. If $ (\lambda x \, t )\,
\sigma \{v/s\}= v\, (\sigma_{i}\{v/s\})\ldots (\sigma_{n}\{v/s\})$,
then $v\, (\sigma_{i}\{v/s\})$ could be a new intuitionistic redex,
when $v=\lambda y\, w$, $v=\langle w_{1}, w_{2}\rangle$, $v=\inj_i
(w)$ or $v= w_0[ y_1. w_{1},y_2.  w_{2}]$. But the complexity of such
a redex is equal to $\tau$.

\item $\langle t_{1}, t_{2}\rangle\, \sigma\, \{v/s\}$, where $\sigma=
\sigma_{1}\ldots \sigma_{n}$ is any case-free stack. By induction
hypothesis, $(*)$ holds for $t_{i} \{v/s\}$ and $\sigma_i\{v/s\}$
where $1\leq i \leq n$. If $ \langle t_{1}, t_{2}\rangle\, \sigma
\{v/s\} = \langle t_{1}\{v/s\}, t_{2}\{v/s\} \rangle\, (\sigma
\{v/s\})$, all the redexes and channel occurrences that we have to
check are in $t_{i}\{v/s\},\sigma \{v/s\}$ or, possibly, the head
redex. The former are dealt with using the inductive hypothesis. As
for the latter, by Lemma~\ref{lem:change_of_value}, the value complexity of
$t_i \{v/s\}$ for $i \in \{1,2\}$ must be at most the value complexity
of $t_i$ or exactly $\tau$, thus (ii) or (iii) holds. If $\langle t_{1},
t_{2}\rangle\, \sigma \{v/s\}= v\, (\sigma_{i}\{v/s\})\ldots
(\sigma_{n}\{v/s\})$, then $v\, (\sigma_{i}\{v/s\})$ could be a new
intuitionistic redex, when $v=\lambda y\, w$, $v=\langle w_{1},
w_{2}\rangle$, $v=\inj_i (w)$ or $v= w_0[ y_1. w_{1},y_2.  w_{2}]$.
But the complexity of such a redex is equal to $\tau$.

\item $\inj _i ( t )\, \sigma\, \{v/s\}$, where $\sigma=
\sigma_{1}\ldots \sigma_{n}$ is any case-free stack.  By induction
hypothesis, $(*)$ holds for $t \{v/s\}$ and $\sigma_j\{v/s\}$ with
$1\leq j \leq n$.  If $ \inj_i (t)\, \sigma \{v/s\} = (\inj_i (t
\{v/s\}))\, (\sigma \{v/s\})$, all the redexes and channel occurrences
that we have to check are either in $t\{v/s \}$ or in $\sigma \{v/s\}$
or, possibly, the head redex, thus the thesis holds.  If $ \inj _i ( t
)\, \sigma\, \{v/s\}= v\, (\sigma_{i}\{v/s\})\ldots
(\sigma_{n}\{v/s\})$, then $v\, (\sigma_{i}\{v/s\})$ could be a new
intuitionistic redex, when $v=\lambda y\, w$, $v=\langle w_{1},
w_{2}\rangle$, $v=\inj_i (w)$ or $v= w_0[ y_1. w_{1},y_2.
w_{2}]$. But the complexity of such a redex is equal to $\tau$.

\item $w_0[z_1.  w_{1}, z_2. w_{2}] \sigma \{v/s\}$, where $\sigma$ is
any case-free stack. By induction hypothesis, $(*)$ holds for $w_{0}
\{v/s\}$, $w_{1} \{v/s\}$, $w_{2} \{v/s\}$ and $\sigma_i\{v/s\}$ for
$1\leq i \leq n$. If 
\begin{align*}
& w_0[z_1.  w_{1}, z_2. w_{2}] \sigma \{v/s\} =\\ &
w_0 \{v/s\} [z_1. w_{1}\{v/s\},z_2.  w_{2}\{v/s\}] (\rho \{v/s\})
\end{align*}
we
first observe that by Lemma~\ref{lem:change_of_value}, the value complexity
of $w_{0} \{v/s\}$ is at most that of $w_{0}$ or exactly $\tau$,
therefore the possible injection or case permutation redex $$w_0
\{v/s\} [z_1. w_{1}\{v/s\},z_2.  w_{2}\{v/s\}]$$ satisfies the thesis.
Again by Lemma~\ref{lem:change_of_value}, the value complexities of $w_{1}
\{v/s\}$ and $w_{2} \{v/s\}$ are respectively at most that of $w_{1}$
and $w_{2}$ or exactly $\tau$. Therefore the complexity of the
possible case permutation redex \\$(w_0[z_1.  w_{1}\{v/s\},
z_2. w_{2}\{v/s\}])\sigma_{1}\{v/s\}$ is either $\tau$, and we are
done, or at most the value complexity of one among $w_{1} , w_{2}$,
thus at most the value complexity of the case permutation redex
$(w_0[z_1.  w_{1}, z_2. w_{2}])\sigma_{1}$ and we are done.

  If $w_0[z_1.  w_{1}, z_2. w_{2}] \sigma \{v/s\} = v\,
(\sigma_{i}\{v/s\})\ldots (\sigma_{n}\{v/s\})$, then there could be a
new intuitionistic redex, when $v=\lambda y\, q$, $v=\langle q_{1},
q_{2}\rangle$, $v=\inj_i (q)$ or $v= q_0[ y_1. q_{1},y_2.  q_{2}]$.
But the complexity of such a redex is $\tau$.

\item $ x\, \sigma \{v/s\}$, where $x$ is any simply typed variable
and $\sigma= \sigma_{1}\ldots \sigma_{n}$ is any case-free stack. By induction
hypothesis, $(*)$ holds for $\sigma_i\{v/s\}$ where $1\leq i \leq
n$. If $ x\, \sigma \{v/s\} = x\, (\sigma \{v/s\})$, all its redexes
and channel occurrences are in $\sigma \{v/s\}$, thus the thesis
holds.  If $ x\, \sigma \{v/s\}= v\, (\sigma_{i}\{v/s\})\ldots
(\sigma_{n}\{v/s\})$, then $v\, (\sigma_{i}\{v/s\})$ could be an
intuitionistic redex, when $v=\lambda y\, w$, $v=\langle w_{1},
w_{2}\rangle$, $v=\inj_i (w)$ or $v= w_0[ y_1. w_{1},y_2.
w_{2}]$. But the complexity of such a redex is equal to $\tau$.

\item $ a\, t\, \sigma \{v/s\}$, where $a$ is a channel variable, $t$
a term and $\sigma= \sigma_{1}\ldots \sigma_{n}$ is any case-free stack.  By
induction hypothesis, $(*)$ holds for $t\{v/s\}, \sigma_i\{v/s\}$ where $1\leq i
\leq n$.

If $ a\, t\, \sigma \{v/s\}= v\, (\sigma_{i}\{v/s\})\ldots
(\sigma_{n}\{v/s\})$, then $v\, (\sigma_{i}\{v/s\})$ could be an
intuitionistic redex, when $v=\lambda y\, w$, $v=\langle w_{1},
w_{2}\rangle$, $v=\inj_i (w)$ or $v= w_0[ y_1. w_{1},y_2.
w_{2}]$. But the complexity of such a redex is equal to $\tau$.
 
If $ a\, t\, \sigma \{v/s\} = a\, (t\{v/s\}) (\sigma \{v/s\})$, in
order to verify the thesis it is enough to check the complexity of the
channel occurrence $a\, (t\{v/s\})$. By Lemma~\ref{lem:change_of_value}, the
value complexity of $t\{v/s\}$ is at most the value complexity of
$t$ or exactly  $\tau$, thus (ii) or (iii) holds.


\item $\pp{a}{t_1 \p \dots \p t_m} 
  \{v/s\}$. By induction hypothesis, $(*)$ holds
for $t_i\{v/s\}$ where $1\leq i \leq m$.  The only redex in $\pp{a}{t_1 \p \dots \p t_m} 
  \{v/s\}$ and not in some  $t_i\{v/s\}$ can be
$\pp{a}{t_1 \{v/s\} \p \dots \p t_m  \{v/s\}} $ itself. But the complexity of such
redex equals the maximal complexity of the channel occurrences of the
form $aw$ occurring in some  $t_i\{v/s\}$, hence it is $\tau$, at most $\tau'$ or equal to the complexity of $\pp{a}{t_1 \p \dots \p t_m} $. \end{itemize}
}
\end{proof}

Here we study what happens after contracting an injection redex.

\begin{lemma}[Eliminate the Case!]
  Let $u$ be a
term in parallel form.  Then for any redex $r$ in $u\{ w_i [t / x_i] /
\inj _i (t)[x_1.w_1, x_2.w_2] \}$ of complexity $\theta$, either $\inj _i (t)[x_1.w_1, x_2.w_2]$ has complexity greater than $\theta$; or
 there is a
redex in $u$ of complexity $\theta$ which belongs to the same group as $r$ or is a case permutation redex.
\end{lemma}
\begin{proof} 
{
Let $v= w_i [t / x_i]$ and $s= \inj _i (t)[x_1.w_1,
x_2.w_2]$. We prove a stronger statement: 
\smallskip

\noindent $(*)$ For any redex $r$ in $u\{ v/s \}$ of complexity
$\theta$, either $\inj _i (t)[x_1.w_1, x_2.w_2]$ has complexity greater than $\theta$; or
 there is a
redex in $u$ of complexity $\theta$ which belongs to the same group as $r$ or is a case permutation redex.  Moreover, for any channel occurrence in $u \{v/s\}$
with complexity $\theta '$,  either $\inj _i (t)[x_1.w_1, x_2.w_2]$ has
complexity greater than $\theta '$, or there is an occurrence of the same
channel with complexity greater or equal than $\theta '$.  \smallskip

The proof is by induction on the size of $u$  and by cases according to the possible shapes of $u$. 
\begin{itemize}
\item $(\lambda x \, t' )\, \sigma\, \{v/s\}$, where $\sigma=
\sigma_{1}\ldots \sigma_{n}$ is any case-free stack. By induction hypothesis,
$(*)$ holds for $t' \{v/s\}$ and $\sigma_i\{v/s\}$ for
$1\leq i \leq n$. If $ (\lambda x \, t')\, \sigma
\{v/s\} = (\lambda x \, t'\{v/s\})\, (\sigma \{v/s\})$, all the redexes
and channel occurrences that we have to check are  in $t'\{v/s\}$, $\sigma
\{v/s\}$ or, possibly, the head redex, thus the thesis holds.
Since $s \neq (\lambda x \, t' )\,
\sigma_1 \dots \sigma _j$, there is no other possible case.

\item $\langle t_{1}, t_{2}\rangle\, \sigma\, \{v/s\}$, where $\sigma=
\sigma_{1}\ldots \sigma_{n}$ is any case-free stack. By induction hypothesis,
$(*)$ holds for  $t_{1} \{v/s\}$, $t_{2} \{v/s\}$ and
$\sigma_i\{v/s\}$ for $1\leq i \leq n$.

If $ \langle t_{1}, t_{2}\rangle\, \sigma \{v/s\} = \langle
t_{1}\{v/s\}, t_{2}\{v/s\} \rangle\, (\sigma \{v/s\})$ all the redexes
and channel occurrences that we have to check are either in $\sigma
\{v/s\}$ or, possibly, the head redex.  By Lemma~\ref{lem:change_of_value}, the
value complexity of $w_i [t / x_i]$ is either at most the value
complexity of $w_{i}$ or the value complexity of $t$. In the first
case, the value complexity of $w_i [t / x_i] $ is at most the value
complexity of $w_{i}$, which is at most the value complexity of $ \inj
_i(t)[x_1.w_1, x_2.w_2] $.  Thus, by Lemma~\ref{lem:change_of_value}, the value
complexities of $ t_{1}\{v/s\}, t_{2}\{v/s\} $ are at most the value complexities respectively of $
t_{1} , t_{2}$, thus the value complexity of $\langle t_{1}\{v/s\}, t_{2}\{v/s\} \rangle$, and hence
that of the possible head redex, is at most the value complexity of $\langle
t_{1} , t_{2} \rangle$ and we are done. In the second case, the value
complexity of $\langle t_{1}\{v/s\}, t_{2}\{v/s\} \rangle$, and hence
that of the head redex, is either at most the value complexity of $\langle
t_{1} , t_{2} \rangle$, and we are done,  or exactly the value complexity of $t$, which is  smaller than the complexity of the injection redex $
\inj_i(t)[x_1.w_1, x_2.w_2]$ occurring in $u$, which is what we wanted
to show. The case in which $\langle t_{1}, t_{2}\rangle\, \sigma
\{v/s\}= v\, (\sigma_{i}\{v/s\})\ldots (\sigma_{n}\{v/s\})$ is
impossible due to the form of $s= \inj _i (t)[x_1.w_1, x_2.w_2]$.

\item $\inj _i ( t' )\, \sigma\, \{v/s\}$, where $\sigma=
\sigma_{1}\ldots \sigma_{n}$ is any case-free stack.  By induction
hypothesis, $(*)$ holds for $t' \{v/s\}$ and $\sigma_j\{v/s\}$ with
$1\leq j\leq n$.  If $ \inj_i (t')\, \sigma \{v/s\} = (\inj_i (t'
\{v/s\}))\, (\sigma \{v/s\})$, all the redexes and channel occurrences
that we have to check are either in $t'\{v/s \}$ or in $\sigma
\{v/s\}$ or, possibly, the head redex, thus the thesis holds.

The
case in which $ \inj_i (t') \, \sigma \{v/s\}= v\,
(\sigma_{i}\{v/s\})\ldots (\sigma_{n}\{v/s\})$ is impossible due to
the form of $s= \inj _i (t)[x_1.w_1, x_2.w_2]$.

\item $ x\, \sigma \{v/s\}$, where $x$ is any simply typed variable
and $\sigma= \sigma_{1}\ldots \sigma_{n}$ is any case-free stack. By induction
hypothesis, $(*)$ holds for $\sigma_i\{v/s\}$ for
$1\leq i \leq n$.  If $ x\, \sigma \{v/s\} = x\, (\sigma \{v/s\})$, all its
redexes and channel occurrences are in $\sigma \{v/s\}$, thus the
thesis holds.  The case in which $ x\, \sigma \{v/s\}= v\, (\sigma_{i}\{v/s\})\ldots
(\sigma_{n}\{v/s\})$ is impossible due to the form of $s= \inj _i
(t)[x_1.w_1, x_2.w_2]$.

 \item $v_0[z_1.  v_{1}, z_2. v_{2}]  \sigma \{v/s\}$, where $\sigma$ is any case-free stack. By induction hypothesis,
$(*)$ holds for  $v_{0} \{v/s\}$, $v_{1} \{v/s\}$, $v_{2} \{v/s\}$ and
$\sigma_i\{v/s\}$ for $1\leq i \leq n$. If
\begin{align*}
& v_0[z_1.  v_{1}, z_2. v_{2}]  \sigma \{v/s\} = \\ & v_0 \{v/s\}
[z_1. v_{1}\{v/s\},z_2.  v_{2}\{v/s\}] (\rho \{v/s\})
\end{align*}
By Lemma~\ref{lem:change_of_value} the
value complexity of $w_i [t / x_i]$ is either at most the value
complexity of $w_{i}$ or the value complexity of $t$. In the first
case, the value complexity of $w_i [t / x_i] $ is at most the value
complexity of $w_{i}$ which is at most the value complexity of $ \inj
_i(t)[x_1.w_1, x_2.w_2] $.  Thus, by Lemma~\ref{lem:change_of_value} the value
complexities of $ v_{0}\{v/s\}, v_{1}\{v/s\}, v_{2}\{v/s\} $  are at most the value complexity respectively of $
v_{0}, v_{1} , v_{2}$. Therefore,
the complexity of the possible case permutation redex $$(v_0[z_1.  v_{1}\{v/s\}, z_2. v_{2}\{v/s\}])\sigma_{1}\{v/s\}$$ is at most the complexity of $v_0[z_1.  v_{1}, z_2. v_{2}]\sigma_{1}$ and we are done. Moreover, the possible injection  or case permutation  redex $$v_0 \{v/s\}
[z_1. v_{1}\{v/s\},z_2.  v_{2}\{v/s\}]$$ satisfies the thesis.  In the second case,  the value
complexities of $ v_{0}\{v/s\}, v_{1}\{v/s\}, v_{2}\{v/s\} $  are respectively at most the value complexities of $
v_{0}, v_{1} , v_{2}$ or exactly the value complexity of $t$. Therefore the complexity of  the possible case permutation redex\\ $(v_0[z_1.  v_{1}\{v/s\}, z_2. v_{2}\{v/s\}])\sigma_{1}\{v/s\}$, is either at most the value complexity of $
v_{1} , v_{2}$, and we are done,  or exactly the value complexity of $t$, which by Proposition \ref{prop:two!} is at most the complexity of the type of $t$, thus is smaller than the complexity of the injection redex $
\inj_i(t)[x_1.w_1, x_2.w_2]$ occurring in $u$. Moreover,  the possible injection  or case permutation  redex $$v_0 \{v/s\}
[z_1. v_{1}\{v/s\},z_2.  v_{2}\{v/s\}]$$ has complexity equal to the value complexity of $v_{0}$ or the value complexity of $t$, and we are done again.\\
 If $v_0[z_1.  v_{1}, z_2. v_{2}]  \sigma \{v/s\}= v\, (\sigma_{i}\{v/s\})\ldots
(\sigma_{n}\{v/s\})$, then  
there could be a new
intuitionistic redex, when $v=\lambda y\, q$, $v=\langle q_{1},
q_{2}\rangle$, $v=\inj_i (q)$ or $v= q_0[ y_1. q_{1},y_2.  q_{2}]$.  If the value complexity of $v=w_i [t / x_i]$ is at most the value complexity of $w_{i}$, then  the complexity of $w_i [t / x_i]
(\sigma_{i}\{v/s\})$ is equal to the complexity of the permutation redex $( \inj _i
(t)[x_1.w_1, x_2.w_2]) \sigma_{i}$. If the value complexity of $v=w_i [t / x_i]$ is  the value complexity of $t$,  by Proposition \ref{prop:two!} the complexity of $w_i [t / x_i]
(\sigma_{i}\{v/s\})$ is at most the complexity of the type of $t$, thus is smaller than the complexity of the injection redex $ \inj _i
(t)[x_1.w_1, x_2.w_2]$  occurring in $u$ and we are done.

\item $ a\, t'\, \sigma \{v/s\}$, where $a$ is a channel variable, $t'$
a term and $\sigma= \sigma_{1}\ldots \sigma_{n}$ is any case-free stack. By
induction hypothesis, $(*)$ holds for  $t'$ and 
$\sigma_i\{v/s\}$ for $1\leq i \leq n$. 
Since $s \neq a\, t'\, \sigma_1 \dots \sigma _j$, the case  $ a\,
t'\, \sigma \{v/s\}= v\, (\sigma_{i}\{v/s\})\ldots
(\sigma_{n}\{v/s\})$ is impossible.

If $ a\, t'\, \sigma \{v/s\} = a\, (t'\{v/s\}) (\sigma \{v/s\})$, in
order to verify the thesis it is enough to check the complexity of the
channel occurrence $a\, (t'\{v/s\})$. By Lemma~\ref{lem:change_of_value}, the
value complexity of $w_i [t / x_i]$ is either at most the value
complexity of $w_{i}$ or exactly the value complexity of $t$. In the first
case, the value complexity of $w_i [t / x_i] $ is at most the value
complexity of $w_{i}$ which is at most the value complexity of $ \inj
_i(t)[x_1.w_1, x_2.w_2] $.  Thus, by Lemma~\ref{lem:change_of_value}, the value
complexity of $t'\{v/s\}$ is at most the value complexity of $t'$ and
we are done. In the second case, the value complexity of $t'\{v/s\}$
is  the value complexity of $t$, which by Proposition
\ref{prop:two!} is at most the complexity of the type of $t$, thus
smaller than the complexity of the injection redex $
\inj_i(t)[x_1.w_1, x_2.w_2]$ occurring in $u$, which is what we wanted
to show.

\item $\pp{a}{t_1 \p \dots \p t_m}  \{v/s\}$. By induction hypothesis, $(*)$ holds for $t_i\{v/s\}$ for $1\leq i \leq m$. The only redex
in  $\pp{a}{t_1 \p \dots \p t_m}  \{v/s\}$ and not in some
$t_i\{v/s\}$ can be  $\pp{a}{t_1 \{v/s\}\p \dots \p t_m \{v/s\}}
\{v/s\}$ itself. But the complexity of such
redex equals the maximal complexity of the occurrences of the channel $a$
 in the  $t_i\{v/s\}$. Hence the
statement follows.
\end{itemize}
}
\end{proof}

Here we study what happens after contracting a case permutation redex.

\begin{lemma}[In Case!]
  Let $u$ be a term in
parallel form. Then for any redex $r_1$ of Group~\ref{group_one} in
$u\{ t[x_1.v_1\xi, x_2.v_2\xi] / t[x_1.v_1,x_2.v_2]\xi \}$, there is a
redex in $u$ with greater or equal complexity than $r_1$; for any
redex $r_2$ of Group~\ref{group_two} in $u\{ t[x_1.v_1\xi, x_2.v_2\xi]
/ t[x_1.v_1,x_2.v_2]\xi \}$, there is a redex of Group~\ref{group_two}
in $u$ with greater or equal complexity than $r_2$.
\end{lemma}
\begin{proof}
{
Let $v= t[x_1.v_1\xi, x_2.v_2\xi]$ and $s= t[x_1.v_1,x_2.v_2]\xi$. We
prove the following stronger statement.
\smallskip

\noindent $(*)$ For any Group~\ref{group_one} redex $r_1$ in $u\{
t[x_1.v_1\xi, x_2.v_2\xi] / t[x_1.v_1,x_2.v_2]\xi \}$, there is a
redex in $u$ with greater or equal complexity than $r_1$; for any
Group~\ref{group_two} redex $r_2$ in $u\{ t[x_1.v_1\xi, x_2.v_2\xi] /
t[x_1.v_1,x_2.v_2]\xi \}$, there is a Group~\ref{group_two} redex in
$u$ with greater or equal complexity than $r_2$.  Moreover, for any channel occurrence in $u \{v/s\}$
with complexity $\theta '$,  there is in $u$ an occurrence of the same
channel with complexity greater or equal than $\theta '$. 
\smallskip

We first observe that the possible Group~\ref{group_one} redexes $v_1 \xi$ and $v_2 \xi$ have at most the
complexity of the case permutation $t[x_1.v_1,x_2.v_2]\xi$. The rest
of the proof is by induction on the shape of $u$.
\begin{itemize}
\item $(\lambda x \, t' )\, \sigma\, \{v/s\}$, where $\sigma=
\sigma_{1}\ldots \sigma_{n}$ is any case-free stack. By induction hypothesis,
$(*)$ holds for $t' \{v/s\}$ and $\sigma_i\{v/s\}$ where $1\leq i \leq
n$.  If $ (\lambda x \, t' )\, \sigma \{v/s\} = (\lambda x \,
t'\{v/s\})\, (\sigma \{v/s\})$, all the redexes and channel occurrences
that we have to check are either in $\sigma \{v/s\}$ or, possibly, the
head redex, thus the thesis holds. There are no more cases since $\sigma$ is case-free.


\item $\langle t_{1}, t_{2}\rangle\, \sigma\, \{v/s\}$, where $\sigma=
\sigma_{1}\ldots \sigma_{n}$ is any case-free stack. By induction hypothesis,
$(*)$ holds for $t_{1} \{v/s\}$, $t_{2} \{v/s\}$ and $\sigma_i\{v/s\}$
where $1\leq i \leq n$.  If $$ \langle t_{1}, t_{2}\rangle\, \sigma
\{v/s\} = \langle t_{1}\{v/s\}, t_{2}\{v/s\} \rangle\, (\sigma
\{v/s\})$$ all the redexes and channel occurrences that we have to
check are either in $\sigma \{v/s\}$ or, possibly, the head redex. The
former are dealt with using the inductive hypothesis. As for the
latter, it is immediate to see that the value complexity of
$t[x_1.v_1\xi, x_2.v_2\xi] $ is equal to the value complexity of $
t[x_1.v_1,x_2.v_2]\xi$. By Lemma~\ref{lem:change_of_value}, the value
complexity of $t_i \{v/s\}$ is at most that of $t_i$, and we are done.
There are no more cases since $\sigma$ is case-free.

\item $\inj_{i}( t' )\, \sigma\, \{v/s\}$, where $\sigma=
\sigma_{1}\ldots \sigma_{n}$ is any case-free stack. By induction hypothesis,
$(*)$ holds for $t' \{v/s\}$ and $\sigma_j\{v/s\}$ where $1\leq j \leq
n$.  If $ \inj_{i}( t' )\, \sigma \{v/s\} = \inj_{i}( t'\{v/s\} )\, (\sigma \{v/s\})$, all the redexes and channel occurrences
that we have to check are either in $\sigma \{v/s\}$ or, possibly, the
head redex, thus the thesis holds. 
There are no more cases since $\sigma$ is case-free.

\item $w_0[z_1.  w_{1}, z_2. w_{2}] \sigma \{v/s\}$, where $\sigma$ is
any case-free stack. By induction hypothesis, $(*)$ holds for $w_{0}
\{v/s\}$, $v_{1} \{v/s\}$, $v_{2} \{v/s\}$ and $\sigma_i\{v/s\}$ for
$1\leq i \leq n$. If
\begin{align*}
& w_0[z_1.  w_{1}, z_2. w_{2}] \sigma \{v/s\} = \\
& w_0 \{v/s\} [z_1. w_{1}\{v/s\},z_2.  w_{2}\{v/s\}] (\sigma \{v/s\})
\end{align*}
Since the value complexity of $t[x_1.v_1\xi, x_2.v_2\xi] $ is equal to
the value complexity of $ t[x_1.v_1,x_2.v_2]\xi$, by
Lemma~\ref{lem:change_of_value} the value complexities of $w_{0} \{v/s\}$,
$w_{1} \{v/s\}$ and $w_{2} \{v/s\}$ are respectively at most that of
$w_{0}$, $w_{1}$ and $w_{2}$. Therefore the complexity of the possible
case permutation redex \\$(w_0\{v/s\}[z_1.  w_{1}\{v/s\},
z_2. w_{2}\{v/s\}])\sigma_{1}\{v/s\}$ is at most the value complexity
respectively of $ w_{1} , w_{2}$, thus the complexity of the case
permutation redex $(w_0[z_1.  w_{1},
z_2. w_{2}])\sigma_{1}$. Moreover, the possible injection or case
permutation redex $$w_0 \{v/s\} [z_1. w_{1}\{v/s\},z_2.
w_{2}\{v/s\}]$$ has complexity equal to the value complexity of
$w_{0}$ and we are done.

 If $w_0[z_1.  w_{1}, z_2. w_{2}]  \sigma \{v/s\}= v\, (\sigma_{i}\{v/s\})\ldots
(\sigma_{n}\{v/s\})$, then there could be a new case permutation redex, because\\ $v=t[x_1.v_1\xi, x_2.v_2\xi]$.
If $\xi$ is case free, by Lemma \ref{lem:stacks_and_zeros}, this redex has
complexity $0$ and we are done;  if not, it has the same complexity as  $t[x_1.v_1, x_2.v_2]
\xi \sigma_{1}$.

\item $ x\, \sigma \{v/s\}$, where $x$ is any simply typed variable
and $\sigma= \sigma_{1}\ldots \sigma_{n}$ is any case-free stack. By
induction hypothesis, $(*)$ holds for $\sigma_i\{v/s\}$ where $1\leq i
\leq n$. If $ x\, \sigma \{v/s\} = x\, (\sigma \{v/s\})$, all its
redexes and channel occurrences are in $\sigma \{v/s\}$, thus the
thesis holds.  
There are no more cases since $\sigma$ is case-free.

\item $ a\, t\, \sigma \{v/s\}$, where $a$ is a channel variable, $t$
a term and  $\sigma= \sigma_{1}\ldots \sigma_{n}$ is any case-free stack. By
induction hypothesis, $(*)$ holds for $t$ and $\sigma_i\{v/s\}$ where
$1\leq i \leq n$.

If $ a\, t\, \sigma \{v/s\} = a\, (t\{v/s\}) (\sigma \{v/s\})$, in
order to verify the thesis it is enough to check the complexity of the
channel occurrence $a\, (t\{v/s\})$.  Since  the value
complexity of $t[x_1.v_1\xi, x_2.v_2\xi] $ is equal to the value
complexity of  \\$ t[x_1.v_1,x_2.v_2]\xi$, by
Lemma~\ref{lem:change_of_value} the value complexity of
$t \{v/s\}$ is at most that of $t$, and we are done. There are no more cases since $\sigma$ is case-free.

\item $\pp{a}{t_1 \p \dots \p t_m } \{v/s\}$. By induction hypothesis,
$(*)$ holds for $t_i\{v/s\}$ where $1\leq i \leq m$ . The only redex
in $\pp{a}{t_1 \p \dots \p t_m } \{v/s\}$ and not in some $t_i\{v/s\}
$ can be $\pp{a}{t_1 \{v/s\} \p \dots \p t_m \{v/s\}} $ itself. But
the complexity of such redex equals the maximal complexity of the
occurrences of the channel $a$ in $t_i\{v/s\}$. Hence the statement
follows.
\end{itemize}
}
\end{proof}

The following result is meant to break the possible loop between the intuitionistic phase and communication phase of our normalization strategy. Intuitively, when Group~\ref{group_one} redexes generate new redexes, these latter have smaller complexity than the former; when Group~\ref{group_two} redexes generate new redexes, these latter does not have worse complexity than the former. 

\begin{proposition}[Decrease!]
  Let $t$ be a term
in parallel form, $r$ be one of its redexes of complexity $\tau$, and
$t'$ be the term that we obtain from $t$ by contracting $r$.
  \begin{enumerate}
  \item If $r$ is a redex of the Group~\ref{group_one}, then the complexity of each
redex in $t'$ is at most
the complexity of a redex of the same group and occurring in $t$; or
is at most the complexity of a case permutation redex occurring in
$t$; or is smaller than $\tau$.

  \item If $r$ is a redex of the Group~\ref{group_two} and not an
activation redex, then every redex in $t'$ either has the complexity
of a redex of the same group occurring in $t$ or has complexity at
most $\tau$.
\end{enumerate} \end{proposition}

\begin{proof}
{
\mbox{}

\noindent $1)$  Suppose $r = (\lambda x^A \, s) \, v$, that $s:B$ and let $q$ be a
redex in $t'$ having different complexity from the one of any redex of
the same group  or the one of any case permutation occurring in $t$. Since $v:A$, we apply Lemma~\ref{lem:replace}
to the term $s [v/x^A]$.  We know that if $q$ occurs in $s [v/x^A]$, since (i) and (ii) do not apply, it has the value complexity of $v$, which by
Proposition~\ref{prop:two!} is at most the complexity of  $A$, which is strictly smaller than+
the complexity of $A\rightarrow B$ and thus than the complexity $\tau$
of $r$. Assume therefore that $q$ does not occur in $s
[v/x^A]$. Since $s [v/x^A] :B$, by applying Lemma~\ref{lem:replace} to
the term $t'=t \{ s [v/x^A]/ (\lambda x^A \, s) \, v\}$ we know that
$q$ has the same complexity as the value complexity of $s [v/x^A]$ -- which by
Proposition~\ref{prop:two!} is at most the complexity of $B$ -- or is a communication redex of complexity equal to the complexity of some channel occurrence $a(w [v/x^A])$ in $s [v/x^A]$, which by Lemma \ref{lem:change_of_value} is at most the complexity of $A$ or at most the complexity of $a w$, and we are done again.

Suppose that $r = \inj _i (s) [x^A.u_{1},y^B.u_{2}]$. By applying
Lemma~\ref{lem:eliminate_the_case}  to 
$$t'=t\{u_{i}[s/x_{i}^{A_{i}}]\, /\, \inj _i (s) [x_{1}^{A_{1}}.u_{1},x_{2}^{A_2}.u_{2}]\}$$ we are done.

 \smallskip

\noindent $2)$ Suppose $r = \langle v_0,v_1\rangle \pi _i$ for $i \in
\{0,1\}$, $\langle v_0,v_1\rangle : A_0 \ET A_1$ let $q$ be a redex in
$t'$ having complexity greater than that of any redex of the same
group in $t$. The term $q$ cannot occur in $v_{i}$, by the assumption
just made. Moreover, by Proposition~\ref{prop:two!}, the value complexity of $v_i$
cannot be greater than the complexity of $A_i$. By applying Lemma~\ref{lem:replace} to the term
$t'=t \{v_i/r\}$, we know that $q$ has complexity equal to the value
complexity of $v_i$, because by the assumption on $q$ the cases (i)
and (ii) of Lemma~\ref{lem:replace} do not apply. Such complexity is
at most the complexity $\tau$ of $r$. Thus we are done.

Suppose that $r = s [x^A.u,y^B.v] \xi $ is a case permutation redex.
By applying Lemma~\ref{lem:in_case} we are done.
 
If $t'$ is obtained by performing a communication permutation, then
obviously the thesis holds.

If $t'$ is obtained by a cross reduction of the form $\pp{a}{ u_1 \p
\dots \p u_m } \mapsto u_{j_1} \p \dots \p u_{j_n} $, for $1 \leq j_1
< \dots < j_n \leq m $, then there is nothing to prove: all
redexes occurring in $t'$ also occur in $t$.


Suppose now that \[r =  \pp{a}{ {\mathcal C}_1 [a^{F_1 \impl G_1 }
    \,t_1] \p \dots \p
{\mathcal C}_m[a^{F_m \impl G_m } \,t_m]}\] 
$t_i =\langle u^i_1 , \dots , u^i_{p_i} \rangle$. 
Then by cross reduction, $r$ reduces
to $\pp{b}{ s_1
\p \dots \p s_m } $ where if $G_i \neq \fal$:
\begin{footnotesize}
  \[ s_{i} \; = \; \pp{a} { {\mathcal C}_1 [a ^{F_1 \impl G_1}\, t_1]
      \dots \p {\mathcal C } _i [ t_j^{b_i \lan \sq{y}_i \ran /
        \sq{y}_j} ] \p \dots \p {\mathcal C}_m[a^{F_m \impl G_m}
      \,t_m]}
  \]
\end{footnotesize}
and if $ G_i = \fal $: 
\begin{footnotesize}
  \[ s_{i} \; = \; \pp{a}{ {\mathcal C}_1 [a^{G_1 \impl G_1}\, t_1]
      \dots \p {\mathcal C}_i [b_i \lan \sq{y}_i \ran ] \p \dots \p
      {\mathcal C}_m [a^{ F_m \impl G_m }\,t_m]}\]
\end{footnotesize}
in which $F_j = G_i$.
and $t'=t\{r'/r\}$.  Let now $q$ be a redex in $t'$ having different
complexity from the one of any redex of the same group in $t$. We
first show that it cannot be an intuitionistic redex: assume it
is. Then it occurs in one of the terms ${\mathcal C }
_i [ t_j^{b_i \lan \sq{y}_i \ran  /
\sq{y}_j} ] $ or $ {\mathcal
    C}_i [b_i \lan \sq{y}_i \ran ] $.
%
%
By
applying Lemma~\ref{lem:replace} to them,
we obtain that the complexity of $q$ is the value complexity of
$t_j^{b_i \lan \sq{y}_i \ran  /
\sq{y}_j} $ or $b_i \lan \sq{y}_i \ran$,
%
%
which, by several applications of Lemma~\ref{lem:change_of_value}, are
at most the value complexity of $t_j$ or  $0$ and
thus by definition at most the complexity of $r$, which is a contradiction.  Assume therefore
that $q= \pp{c}{p_1 \p \dots \p p_z }$ 
is a communication redex. 
Every channel occurrence of $c$ in the terms  ${\mathcal C }
_i [ t_j^{b_i \lan \sq{y}_i \ran  /
\sq{y}_j} ] $ or $ {\mathcal
    C}_i [b_i \lan \sq{y}_i \ran ] $
is of the form $ cw\{ t_j^{b_i \lan \sq{y}_i \ran  /
\sq{y}_j} / a\, t_i\} $ or $ cw\{ b_i \lan \sq{y}_i \ran   / a\,
t_i\} $ 
where $cw$ is a channel occurrence in $t$. By Lemma \ref{lem:change_of_value}, each of those occurrences has either at most the value complexity of $cw$ or has at most the value complexity of $r$, which is a contradiction. 
}
\end{proof}

The following result is meant to break the possible loop during the communication phase: no new activation is generated after a cross reduction, when there is none to start with.



\begin{lemma}[Freeze!]
  Suppose that $s$ is a term in
parallel form that does not contain projection nor case permutation
nor activation redexes.  Let $\pp{a}{q_1 \p \dots \p q_m }$
be some redex of $s$ of
complexity $\tau$. If $s'$ is obtained from $s$ by performing first a
cross reduction on $\pp{a}{q_1 \p \dots \p q_m }$
and then contracting all projection
and case permutation redexes, then $s'$ contains no activation
redexes. \end{lemma}
 \begin{proof} 
{
Let $ \pp{a}{q_1 \p \dots \p q_m } = $ \\$
\pp{a}{ {\mathcal
         C}_1 [a^{F_1 \impl G_1 } \,t_1 \sigma _1 ] \p \dots \p
{\mathcal C}_m[a^{F_m \impl G_m } \,t_m \sigma _m ]}$
%
%
where $\sigma_i$ for $1 \leq i \leq m $ are the stacks
which are applied to $a\, t_i$, and $t$ be the cross reduction
redex occurring in $s$ that we reduce to obtain $s'$. Then after performing the
cross reduction and contracting all the intuitionistic redexes,
$t$ reduces to $\pp{b}{ s_1 \p
\dots \p s_m } $
where if $G_i \neq \fal$:
\begin{footnotesize}
  \[ s_{i} \; = \; \pp{a} { {\mathcal C}_1 [a ^{F_1 \impl G_1}\, t_1]
      \dots \p {\mathcal C } _i [ t_j ' ] \p \dots \p {\mathcal C}_m[a^{F_m \impl G_m}
      \,t_m]}
  \]
\end{footnotesize}
and if $ G_i = \fal $: 
\begin{footnotesize}
  \[ s_{i} \; = \; \pp{a}{ {\mathcal C}_1 [a^{G_1 \impl G_1}\, t_1]
      \dots \p {\mathcal C}_i [b_i \lan \sq{y}_i \ran ] \p \dots \p
      {\mathcal C}_m [a^{ F_m \impl G_m }\,t_m]}\]
\end{footnotesize}
in which $F_j = G_i$.
%
%
where $t_j '$ are the terms obtained reducing all
projection and case permutation redexes respectively in $t_j ^{b_i \lan \sq{y}_i \ran /
        \sq{y}_j}$. Moreover $s'=s\{t'/t\}$.
 
We observe that $a$ is active and hence the terms $s_i$ for $1 \leq i \leq m $
are not activation
redexes. Moreover, since all occurrences of $b$ are of the form $b_i
\langle \sq{y}_i\rangle$, $t'$ is not an
activation redex.  Now we consider the
channel occurrences in any term $s_i$. We first show
that there are no  activable channels in $t_j ' $ which are bound in $s'$. For any
subterm $w$ of $t_j$, since for any  stack $\theta$ of projections, $b_i \langle
\sq{y}_i\rangle\theta $ has value complexity $0$, we can apply repeatedly
Lemma~\ref{lem:change_of_value} to $w ^{ b_i \langle \sq{y}_i\rangle / \sq{y}_j}$
and obtain that the value complexity of $w ^{ b_i \langle
  \sq{y}_i\rangle / \sq{y}_j}$
is exactly the value complexity of $w$. This implies that
there is no activable channel 
in $t_j ^{ b_i \langle
  \sq{y}_i\rangle / \sq{y}_j}$ because there is none in $t_j$.   
The following general statement immediately implies that there is no activable
channel in $t_j'$ which is bound in $s'$ either.  \smallskip

\noindent $(*)$ Suppose that $r$ and $\theta $ are respectively a
simply typed $\lambda$-term and a stack contained in $s'$ that do not contain projection
and permutation redexes, nor  activable channels bound in $s'$.  If $r'$ is obtained from $r\theta$
by performing all possible projection and case permutation reductions,
then there are no  activable channels in $r'$ which are bound in $s'$.
\smallskip

Proof. By induction on the size of $r$. We proceed by cases according to the
shape of $r$.
\begin{itemize}
\item If $r=\lambda x w$, $r=\inj_{i}(w)$, $r=w \exfalso_{P} $, $r=x$ or
$r=a w$, for a channel $a$, then $r'=r\theta$ and the thesis holds.

\item If $r=\langle v_{0}, v_{1}\rangle$ the only redex that can occur
in $r\theta$ is a projection redex, when $\theta= \pi_{i}
\rho$. Therefore, $r\theta \mapsto v_{i}\rho\mapsto^{*} r'$.  By
induction hypothesis applied to $v_{i}\rho$, there are no activable channels in
$r'$ which are bound in $s'$.

\item If $r=t [x.v_0, y.v_1]$, then \\$r\theta \mapsto^{*}
t[x.v_0\theta, y.v_1\theta ]\mapsto^{*} t[x.v_0', y.v_1']= r'$.  By
induction hypothesis applied to $v_0'$ and $v_1'$, there are no activable
channels in $v_0'$ and $v_1'$ and we are done.

\item If $r=p \nu \xi$, with $\xi$ case free, then $r'=p\nu \xi$ and
the thesis holds.
\end{itemize}

Now, let $c$ be any non-active channel bound in $s'$ occurring in
${\mathcal C } _i [ t_j ' ] $ or $ {\mathcal C}_i [b_i \lan \sq{y}_i \ran
]$
but not in $u'$: any of its occurrences is of the form $c \langle p_1
, \dots , p_l \{u' / av \rho \}, \dots , p_n \rangle $, where each
$p_l$ is not a pair. We want to show that the value complexity of
$p_l \{u' /av \rho \} $ is exactly the value complexity of
$p_l$. Indeed, $p_l = r \, \nu $ where $\nu$ is
  a case-free stack. If $r$ is of the form $\lam x w $, $\lan q_1 ,
q_2 \ran $, $\inj _i (w)$, $x$, $d w$, with $d\neq a$, then the value
complexity of $p_l [u' /av \rho ] $ is the same as that of $p_l$
(note that if $r=\lan q_{1}, q_{1}\ran$, then $\nu$ is not empty). If
$r= v_{0}[x_{1}.v_{1}, x_{2}.v_{2}]$, then $\nu$ is empty, otherwise
$s$ would contain a permutation redex, so $c \langle p_1
, \dots , p_l, \dots , p_n \rangle $
is activable, and there is an activation
redex in $s$, which is contrary to our assumptions. The case $r=a v $,
$\nu=\rho \rho'$ is also impossible, otherwise  $c \langle p_1
, \dots , p_l, \dots , p_n \rangle $ would be activable, and we are done.
}
\end{proof}

\begin{definition} We define the height $\mathsf{h}(t)$ of a term
$t$ in parallel form as
    \begin{itemize}
    \item $\mathsf{h}( u ) \; = \; 0$ if $u$ is simply typed
$\lambda$-term
    \item $\mathsf{h}( u \parallel _a v ) \; = \; 1+ \,
\text{max}(\mathsf{h}( u ) , \mathsf{h} ( v ) )$
\end{itemize}
  \end{definition}

The communication phase of our reduction strategy is finite.

\begin{lemma}[Communicate!]
  Let $t$ be any term in parallel form that does not contain
projection, case permutation, or activation redexes.  Assume moreover
that all redexes in $t$ have complexity at most $\tau$.  Then $t$ reduces to a term
containing no redexes, except Group~\ref{group_one} redexes of complexity at most $\tau$.
\end{lemma}
\begin{proof} 
{
We prove the statement by lexicographic induction on the triple $( n,
h,g )$ where
\begin{itemize}
   \item $n$ is the number of subterms $\pp{a}{u_1 \p \dots \p u_m }$
     of $t$ such that  $\pp{a}{u_1 \p \dots \p u_m }$ is an active, but not uppermost, session.
\item $h$ is the function mapping each natural number $m\geq 2$
into the number of uppermost active sessions in $t$ with height $m$.
   \item $g$ is the function mapping each natural number $m$
into the number of uppermost active sessions  $\pp{a}{u_1 \p \dots \p u_m }$ in $t$ containing $m$ occurrences of $a$.
\end{itemize}

We employ the following lexicographic ordering between functions for the second and third
elements of the triple: $f<f'$ if and only if there is some $i$ such
that for all $j>i$,  $f(j) = f'(j) = 0$
and there is some $i$ such
that for all $j>i$, $f'(j)=f(j)$ and $f(i)<f'(i)$.

If $h(j) > 0$, for some $j\geq 2$, then there is at least an active
session $\pp{a}{u_1 \p \dots \p u_m }$ in $t$ that does not contain
any active session and such that $\mathsf{h}( \pp{a}{u_1 \p \dots \p
u_m }) = j $. Hence $u_i= \pp{b}{s_1 \p \dots \p s_q }$ for some $1
\leq i \leq i \leq m$. We obtain $t'$ by applying inside $t$  the
permutation  $ \pp{a}{ u_1\p \dots \p \pp{b}{s_1\p  \dots \p s_q}  \dots \p
u_m} \; \mapsto \;   \pp{b}{ \pp{a}{ u_1\p \dots \p s_1\dots \p u_m}
\dots \p  \pp{a}{ u_1\p \dots \p s_q \dots \p u_m}}  $
We claim that the term $t'$ thus obtained has complexity  has
complexity $( n, h', g')$, with $h'<h$. Indeed,  $\pp{a}{u_1 \p \dots
  \p u_m }$  does not
contain active sessions, thus $b$ is not active and the number of
active sessions which are not uppermost in $t'$ is still $n$.  With
respect to $t$, the term $t'$ contains one less uppermost active
session with height $j$ and $q$ more of height $j-1$, and therefore
$h' < h$. Furthermore, since the permutations do not change at all the
purely intuitionistic subterms of $t$, no new activation or
intuitionistic redex is created. In conclusion, we can apply the
induction hypothesis on $t'$ and thus obtain the thesis.

If $h(m)=0$ for all $m\geq 2$, then let us consider an uppermost
active session $\pp{a}{u_1 \p \dots \p u_m }$ in $t$ such that
$\mathsf{h}(\pp{a}{u_1 \p \dots \p u_m }) = 1 $; if there is not, we
are done. We reason by cases on the distribution of the occurrences of
$a$. Either (i) some $u_i$ for $1 \leq i \leq m$ does not contain any
occurrence of $a$, or (ii) all $u_i$ for $1 \leq i \leq m$ contain
some occurrence of $a$.



Suppose that (i) is the case and, without loss of generality, that $a$
occurs $j$ times in $u$ and does not occur in $v$. We then obtain a
term $t'$ by applying a cross reduction $\pp{a}{ u_1 \p \dots \p u_m }
\mapsto u_{j_1} \p \dots \p u_{j_p} $. If there is an active session
$\pp{b}{ s_1 \p \dots \p s_q }$ in $t$ such that $\pp{a}{ u_1 \p \dots
\p u_m } $ is the only active session contained in some $s_i$ for
$1\leq i \leq q$, then the term $t'$ has complexity $(n-1, h', g' )$,
because $\pp{b}{ s_1 \p \dots \p s_q }$ is an active session which is
not uppermost in $t$, but is uppermost in $t'$; if not, we claim that
the term $t'$ has complexity $( n, h, g' )$ where $g'<g$. Indeed,
first, the number of active sessions which are not uppermost does not
change. Second, the height of all other uppermost active sessions does
not change. Third, $g'(j) = g(j)-1$ and, for any $i\neq j$, $g'(i) =
g(i)$ because, obviously, no channel belonging to any uppermost active
session different from $\pp{a}{ u_1 \p \dots \p u_m }$ occurs in
$u$. Since the reduction $\pp{a}{ u_1 \p \dots \p u_m } \mapsto
u_{j_1} \p \dots \p u_{j_p} $ does not introduce any new
intuitionistic or activation redex, we can apply the induction
hypothesis on $t'$ and obtain the thesis.

Suppose now that (ii) is the case and that all $u_i$ for $1\leq i \leq
m $ together contain $j$ occurrences of $a$. Then $\pp{a}{ u_1 \p
\dots \p u_m }$ is of the form $\pp{a}{ {\mathcal C}_1 [a^{F_1 \impl
G_1 } \,t_1] \p \dots \p {\mathcal C}_m[a^{F_m \impl G_m } \,t_m]}$
where $a$ is active, ${\mathcal C}_j [a^{ F_j \impl G_j } \,t_j] $ for
$1 \leq j \leq m $ are simply typed $\lambda$-terms; $a^{ F_j \impl
G_j } $ is rightmost in each of them.
%
%
%
Then we can apply the cross reduction $\pp{a}{ {\mathcal C}_1 [a^{F_1 \impl G_1 } \,t_1] \p \dots \p
{\mathcal C}_m[a^{F_m \impl G_m } \,t_m]} \; \mapsto \; \pp{b}{ s_1 \p
\dots \p s_m } $
in which $b$ is fresh and for $1 \leq i \leq m $, we define, if $G_i \neq \fal$:
$ s_{i} \; = \; 
\pp{a} { {\mathcal C}_1 [a ^{F_1 \impl G_1}\, t_1] \dots \p  {\mathcal C }
_i [ t_j^{b_i \lan \sq{y}_i \ran  /
\sq{y}_j} ] \p \dots \p {\mathcal C}_m[a^{F_m \impl G_m} \,t_m]}$ and
if $G_i = \fal$ : $  s_{i} \; = \; 
\pp{a}{ {\mathcal C}_1 [a^{G_1 \impl G_1}\, t_1] \dots \p {\mathcal
    C}_i [b_i \lan \sq{y}_i \ran ] \p \dots \p 
{\mathcal C}_m [a^{ F_m \impl G_m }\,t_m]}$
where $F_j = G_i$; $\sq{y}_z$ for $1 \leq z \leq m$ is the sequence of the free
variables of $t_z$ bound in $\mathcal{C}_z[a^{F_z \impl G_z}
\, t_z]$; $b_i=b^{B_{i}\IMPL B_{j}}$, where $B_{z}$ for $1 \leq z \leq
m$ is the type of $\lan \sq{y}_z \ran$. By
Lemma~\ref{lem:freeze}, after performing all projections and case
permutation reductions in all $  {\mathcal C }
_i [ t_j^{b_i \lan \sq{y}_i \ran  /
\sq{y}_j} ] $ and  $ {\mathcal
    C}_i [b_i \lan \sq{y}_i \ran ] $ for $1 \leq i\leq m$
we obtain a term $t'$ that contains no activation redexes; moreover,
by Proposition~\ref{prop:decrease}.2., $t'$ contains only
redexes having complexity at most $\tau$.

We claim that the term $t'$ thus obtained has complexity $\langle n,
h, g' \rangle$ where $g'<g$.  Indeed, the value $n$ does not change
because all newly introduced occurrences of $b$ are not active. The new active sessions $ s_{i}  =  \pp{a} { {\mathcal C}_1 [a ^{F_1
\impl G_1}\, t_1] \dots \p {\mathcal C } _i [ t_j^{b_i \lan \sq{y}_i
\ran / \sq{y}_j} ] \p \dots \p {\mathcal C}_m[a^{F_m \impl G_m}
\,t_m]}$ or $ s_{i} = \pp{a}{ {\mathcal C}_1
[a^{G_1 \impl G_1}\, t_1] \dots \p {\mathcal C}_i [b_i \lan \sq{y}_i
\ran ] \p \dots \p {\mathcal C}_m [a^{ F_m \impl G_m }\,t_m]}$ for $1
\leq i \leq m$
%
%
have all height $1$ and contain $j-1$ occurrences of $a$. Since
furthermore the reduced term does not contain channel occurrences of
any uppermost active session different from $\pp{a}{ u_1 \p \dots \p
u_m }$ we can infer that $g'(j) = g(j)-1$ and that, for any $i$ such
that $i > j$, $g'(i) = g(i)$.

We can apply the induction hypothesis on $t'$ and obtain the thesis.
}
\end{proof}

We now combine together all the main results achieved so far.

\begin{proposition}[Normalize!]
  Let $t: A$ be any term
in parallel form. Then $t \mapsto^{*} t'$, where $t'$ is a
parallel normal form.
\end{proposition}
\begin{proof} 
{
Let $\tau$ be the maximum among the complexity of the redexes in $t$. We prove
the statement by induction on $\tau$.

 Starting from $t$, we reduce all intuitionistic redexes and obtain a
term $t_1$ that, by Proposition~\ref{prop:decrease}, does not contain
redexes of complexity greater than $\tau$. By Lemma~\ref{activation},
$t_1\mapsto ^* t_2$ where $t_2$ does not contain any redex, except
cross reduction redexes of complexity at most $\tau$.  By
Lemma~\ref{lem:com_phase}, $t_2\mapsto^* t_3$ where $t_3$ contains
only Group~\ref{group_one} redexes of complexity at most $\tau$. Suppose
$t_3\mapsto^* t_4$ by reducing all Group~\ref{group_one} redexes, starting from
$t_{3}$. By Proposition~\ref{prop:decrease}, every Group~\ref{group_one} redex
which is generated in the process has complexity at most $\tau$, thus
every Group~\ref{group_two} redex which is generated has complexity smaller than
$\tau$ , thus $t_{4}$ can only contain redexes with complexity smaller
than $\tau$. By induction hypothesis $t_4 \mapsto ^* t'$, with $t'$ in
parallel normal form.
}
\end{proof}

  The normalization for $\lama$, now easily follows.

\begin{theorem}
  Suppose that  $ t: A$ is a proof term  of $\LC$. Then $t\mapsto^{*} t': A$, where $t'$ is a normal parallel form.
\end{theorem}

\section{ 
The Subformula Property}

We now show that normal $\lama$-terms satisfy the
Subformula Property:  a normal proof does not
contain concepts that do not already appear in the premisses or in the
conclusion. This, in turn,
implies that our Curry--Howard correspondence for $\lama$ is meaningful
from the logical perspective and produces analytic  proofs.

\begin{proposition}[Parallel Normal Form Property]
If $t\in
\nf$ is a $\lama$-term, then it is in parallel form.
\end{proposition}
\begin{proof}
{
By induction on $t$. 
\begin{itemize}

\item $t$ is a variable $x$. Trivial. 

\item $t=\lambda x\, v$. Since $t$ is normal, $v$ cannot be of the
form $\pp{a}{u_1 \p \dots \p u_m }$, otherwise one could apply the
permutation \[t= \lam x^{A} \, \pp{a}{u_{1}\p\dots \p u_{m}} \mapsto
\pp{a}{\lam x^{A} \, u_{1}\p\dots \p \lam x^{A} . u_{m}} \] and $t$
would not be in normal form. Hence, by induction hypothesis $v$ must
be a simply typed $\lambda$-term.

\item $t=\langle v_{1}, v_{2}\rangle$. Since $t$ is normal, neither
$v_{1}$ nor $v_{2}$ can be of the form $\pp{a}{u_1 \p \dots \p u_m }$,
otherwise one could apply one of the permutations
\[\langle \pp{a}{u_{1}\p \dots\p u_{m}},\, w\rangle \mapsto
\pp{a}{\langle u_{1}, w\rangle\p \dots\p \lan u_{m}, w\ran}\]
\[\langle w, \,\pp{a}{u_{1}\p \dots\p u_{m}}\rangle \mapsto
\pp{a}{\lan w, u_{1}\ran\p \dots\p \lan w, u_{m} \ran}\] and $t$ would
not be in normal form. Hence, by induction hypothesis $v_{1}$ and
$v_{2}$ must be simply typed $\lambda$-terms.

\item $t=v_1 \, v_2$. Since $t$ is normal, neither $v_1$ nor $v_2$ can
be of the form $\pp{a}{u_1 \p \dots \p u_m }$, otherwise one could
apply one of the permutations \[\pp{a}{u_{1}\p\dots \p u_{m}}\, w
\mapsto \pp{a}{u_{1}w\p\dots \p u_{m}w} \] \[w\, \pp{a}{u_{1}\p\dots
\p u_{m}} \mapsto \pp{a}{w u_{1}\p\dots \p w u_{m}}\] and $t$ would
not be in normal form. Hence, by induction hypothesis $v_{1}$ and
$v_{2}$ must be simply typed $\lambda$-terms.

\item $t=  \efq{P}{v}$. Since $t$ is normal, $v$ cannot
be of the form $\pp{a}{u_1 \p \dots \p u_m }$, otherwise one could apply the permutation
\[\efq{P}{\pp{a}{u_{1}\p\dots\p u_{m}}} \mapsto
\Ecrom{a}{\efq{P}{w_{1}}}{\efq{P}{w_{2}}}\] and $t$ would not be in
normal form. Hence, by induction hypothesis $u_{1}$ and $u_{2}$ must
be simply typed $\lambda$-terms.

\item $t=u\, \pi_{i}$. Since $t$ is normal, $v$ can
be of the form $\pp{a}{u_1 \p \dots \p u_m }$, otherwise one could apply the permutation
\[\pp{a}{u_{1}\p \dots\p u_{m}}\,\pi_{i} \mapsto \pp{a}{u_{1}\pi_{i}\p
\dots\p u_{m}\pi_{i}}\] and $t$ would not be in normal form. Hence, by
induction hypothesis $u$ must be a simply typed $\lambda$-term, which
is the thesis.

\item $t= \pp{a}{u_1 \p \dots \p u_m }$. By induction hypothesis the thesis holds
for $u_i$ where $1 \leq i \leq m$ and hence trivially for $t$.
\end{itemize}
}
\end{proof}

\begin{theorem}[Subformula Property]
  Suppose
\[x_{1}^{A_{1}}, \ldots, x_{n}^{A_{n}}, a_{1}^{D_{1}}, \ldots, a_{m}^{D_{m}}\vdash t: A \quad \mbox{and} \quad
t\in \nf. \quad \mbox{Then}:\]
\begin{enumerate}
\item 
For each channel variable $a^{B \IMPL C}$ occurring bound in
$t$, the prime factors of $B,C$ are
subformulas of $A_{1}, \ldots, A_{n}, A$ or proper subformulas of $D_{1}, \ldots, D_{m}$.
\item 
 The type of any subterm of $t$ 
 is either a subformula or a conjunction of subformulas of
$A_{1}, \ldots, A_{n}, A$ and of proper subformulas of $D_{1}, \ldots, D_{m}$.
\end{enumerate}
\end{theorem}
\begin{proof} 
{
We proceed by structural induction on $t$ and reason by cases, according to the form of $t$.

\begin{itemize}
\item 

$t = \pair{u}{v} : F \ET G $. Since $t\in \nf$, by Proposition \ref{proposition-parallelform} it is in parallel form, thus is a simply typed $\lambda$-term. Therefore no communication variable can be
bound inside $t$, thus 1. trivially holds. By induction hypothesis, 2. holds for $u : F$ and $v : G$.  Therefore, the type of any subterm of $u$ is either a subformula or a conjunction
of subformulas of  $A_{1}, \ldots, A_{n}$, of $F$ and of proper
subformulas of $D_{1}, \ldots, D_{m}$ and any subterm of $v$ is
either a subformula or a conjunction of subformulas of some $A_{1},
\ldots, A_{n}$, of $G$ and of proper
subformulas of $D_{1}, \ldots, D_{m}$.  Moreover, any subformula of
$F$ and $G$ must be a subformula of the type $F \ET G$ of $ t $. Hence the type of any subterm of
$\pair{u}{v}$ is either a subformula or a conjunction of subformulas
of $A_{1}, \ldots, A_{n}, F \ET G$ or a proper subformula of $D_{1},
\ldots, D_{m}$ and the statement holds for $t$ as well.

\item $t = \lambda x^F \, u : F \IMPL G$. Since $t\in \nf$, by Proposition \ref{proposition-parallelform} it is in parallel form, thus is a simply typed $\lambda$-term. Therefore no communication variable can be
bound inside $t$, thus 1. trivially holds. By induction hypothesis, 2. holds for $u:G$. Therefore the type of any
subterm of $u$ is either a subformula or a conjunction of subformulas
of some $A_{1}, \ldots, A_{n}, F$, of $G$ and of proper subformulas of
$D_{1}, \ldots, D_{m}$. Since the type $F$ of $x$ is a subformula of $F
\IMPL G$,  the type of any subterm of $\lambda x^F \, u$ is either
a subformula or a conjunction of subformulas of $A_{1}, \ldots, A_{n},
F \IMPL G$ or a proper subformula of $D_{1}, \ldots, D_{m}$ and the
statement holds for $t$ as well.

\item $t = \inj_{i}({u}) : F \VEL G$ for $i \in \{0,1\}$. Without loss
  of generality assume that $i=1$ and  $u:G$. Since $t\in \nf$, by Proposition \ref{proposition-parallelform} it is in parallel form, thus is a simply typed $\lambda$-term. Therefore no communication variable can be
bound inside $t$, thus 1. trivially holds. By induction hypothesis, 2. holds for $u:F$. Therefore,  the type of any
subterm of $u$ is either a subformula or a conjunction of subformulas
of some $A_{1}, \ldots, A_{n}$, of $F$ or proper subformulas of
$D_{1}, \ldots, D_{m}$. Moreover, any subformula of $G$ must be a
subformula of the type $ F \VEL G$ of $t$.  Hence the type of any subterm of $\inj_{i}({u})$ is either
a subformula or a conjunction of subformulas of $A_{1}, \ldots, A_{n},
F \VEL G$ or a proper subformula of $D_{1}, \ldots, D_{m}$ and the
statement holds for $t$ as well.

\item $t = x^{A_{i}} \, \sigma : A$ for some $A_{i}$ among $A_{1},
\ldots, A_{n}$ and stack $\sigma$. Since $t\in \nf$, it is in parallel form, thus is a simply typed $\lambda$-term and no communication variable can be
bound inside $t$. By induction hypothesis, for any element $\sigma_j :
S_j$ of $\sigma$, the type of any subterm of $\sigma_j$ is either a
subformula or a conjunction of subformulas of some $A_{1}, \ldots,
A_{n}$, of the type $S_j$ of $\sigma_j$ and of proper subformulas of
$D_{1}, \ldots, D_{m}$.  

If $\sigma$ is case-free, then every $S_j$ is a
subformula of  $A_i$,  or of $A$, when $\sigma=\sigma' \exfalso_{A}$. Hence, the type of any subterm of
$x^{A_{i}} \, \sigma$ is either a subformula or a conjunction of
subformulas of $A_{1}, \ldots, A_{n}, A$ or of proper subformulas of
$D_{1}, \ldots, D_{m}$ and the statement holds for $t$ as well. 

In case $\sigma$ is not case-free, then, because of case permutations,  $\sigma = \sigma ' [y^G. v_1 ,z^E. v_2] $, with $\sigma'$ case-free.  By induction hypothesis we know that the type of any subterm of
$v_1 : A$ or $v_2:A$ is either a subformula or a conjunction of
subformulas of some $A_{1}, \ldots, A_{n}$, of $A, G, E$ and of proper
subformulas of $D_{1}, \ldots, D_{m}$. Moreover, $G$ and $E$ are
subformulas of $A_{i}$ due to the properties of stacks. Hence, the
type of any subterm of $x \, \sigma ' [y^G. v_1 ,z^E. v_2]$ is either
a subformula or a conjunction of subformulas of $A_{1}, \ldots, A_{n},
A$ and of  proper subformulas of $D_{1}, \ldots, D_{m}$ and also in this
case the statement holds for $t$ as well.

\item $t = a^{D_{i}}u \, \sigma : A$ for some $D_{i}$ among $D_{1},
\ldots, D_{n}$ and stack $\sigma$. As in the previous case.

\item $t = \pp{b}{u_{1}\p\dots  \p  u_{k}} : A$ and $b^{G_{i}\rightarrow H_{i}}$ occurs in
$u_{i}$.  Suppose, for the sake of contradiction, that the statement does not
hold. We know by induction hypothesis that the statement holds for $u_{1}: A, \dots, u_{k}: A$. We
first show that it cannot be the case that
\begin{quote}
  $(*)$ all prime factors of $G_{1}, H_{1}, \dots, G_{k}, H_{k}$
  are subformulas of $A_{1}, \ldots, A_{n}, A$ or proper subformulas
  of $D_{1}, \ldots, D_{m}$.
\end{quote}

Indeed, assume by contradiction that $(*)$ holds. Let us consider the type $T$ of any subterm of $t$ which is not a bound
communication variable and the formulas $B,C$ of any bound
communication variable $a^{B \IMPL C}$ of $t$. Let $P$ be any prime factor of $T$ or $B$ or $C$. By induction hypothesis applied to $u_{1},\dots, u_{n}$, we obtain that $P$  is either subformula
or conjunction of subformulas of $A_{1}, \ldots, A_{n}, A$ and of
proper subformulas of $D_{1}, \ldots, D_{m}, G_{1}, H_{1}, \dots,
G_{k}, H_{k}$.  Moreover,  $P$ is prime and so it must be subformula of
$A_{1}, \ldots, A_{n}, A$ or a proper subformula of $D_{1}, \ldots, D_{m}$ or a prime factor of $G_{1}, H_{1}, \dots, G_{k}, H_{k}$. Since $(*)$ holds, $P$ must be a subformula of $A_{1}, \ldots, A_{n}, A$ or proper subformula of $D_{1}, \ldots, D_{m}$, and this contradicts the assumption that the subformula property does not hold for $t$.

We shall say from now on that any bound channel variable $a^{F_1 \IMPL
F_2}$ of $t$ \emph{violates the subformula property maximally (due to $Q$)} if
(i) some prime factor $Q$ of $F_1$ or $F_2$ is neither a
subformula of $A_{1}, \ldots, A_{n}, A$ nor a proper subformula of
$D_{1}, \ldots, D_{m}$ and (ii) for every other bound channel variable
$c^{S_{1}\IMPL S_{2}}$ of $t$, if some prime factor $Q'$ of $S_1$ or
$S_2$ is neither a subformula of $A_{1}, \ldots, A_{n}, A$ nor a
proper subformula of $D_{1}, \ldots, D_{m}$, then $Q'$ is complex at
most as $Q$. If $Q$ is a subformula of $F_1$ we say that  $a^{F_1 \IMPL
F_2}$  violates the subformula property maximally \emph{in the input}.

It follows from $(*)$ that a channel variable maximally violating the
subformula property must exist. We show now that there also exists
a subterm $c^{F_1 \IMPL F_2} w$ of $t$ such that $c$ maximally
violates the subformula property in the input due to $Q$, and $w$ does not
contain any channel variable that violates the subformula property
maximally.

In order to prove the existence of such term, we prove 
\smallskip

\noindent $(**)$ Let $t_{1}$ be any subterm of $t$ 
such that $t_{1}$ contains at least a maximally violating channel and all maximally violating channel of $t$ that are free in $t_{1}$ are maximally violating in the input. Then  there is a simply typed  subterm $s$ of $t_{1}$ 
 such that $s$ contains at least a maximally violating channel, and such that all
occurrences of maximally violating channels occurring in $s$  violate the subformula property in the input.
\smallskip


We proceed by induction on the number $n$ of $\parallel$
operators that occur  in $t_1$.  

If $n=0$,  it is enough to pick $s=t_{1}$. 

If $n>0$, let $t_{1}=\pp{d}{v_{1}\p\dots\p v_{n}}$ and assume $d^{E_{i}\IMPL F_{i}}$ occurs in $v_{i}$. If no $d^{E_{i}\IMPL F_{i}}$ maximally violates the subformula property, we obtain the thesis by applying the induction hypothesis to any $v_{i}$. Assume therefore that some $d^{E_{i}\IMPL F_{i}}$ maximally violates the subformula property due to $Q$. Then there is some $d^{E_{j}\IMPL F_{j}}$ such that  $Q$ is a prime factor of $E_{i}$ or $E_{j}$. By induction hypothesis applied to $u_{i}$ or $u_{j}$, we obtain the thesis.

\smallskip 

By $(**)$ we can infer that in $t$ there is a simply typed $\lambda$-term
$s$  that contains at least one occurrence of a maximally violating
channel and only occurrences of maximally violating channels that are
maximally violating in the input. The rightmost of the maximally violating channel occurrences in $s$ is then of the form $c^{F_1 \IMPL
F_2} w$ where $c$ maximally violates the subformula property in the
input and  $w$ does not contain any channel variable maximally
violating the subformula property.

Consider now this term  $c^{F_1 \IMPL F_2} w$.

Since $Q$ is a prime factor of $F_1$, it is either an atom $P$ or a
formula of the form $Q' \IMPL Q''$ or of the form $Q' \VEL Q''$.

Let $w=\langle w_{1}, \ldots, w_{j}\rangle$, where each $w_{i}$ is not
a pair, and let $k$ be such that that $Q$ occurs in the type of
$w_{k}$.

We start by ruling out the case that $w_{k}= \lambda y\, s$ or $w_{k}=
\inj _i (s)$ for $i \in \{0,1\}$, otherwise it would be possible to
perform an activation reduction or a cross reduction to some subterm
$u'\parallel_{c} v'$, which must exist since $c$ is bound. 

Suppose now, by contradiction, that $w_{k} = x^T \, \sigma$ where
$\sigma$ is a stack.  It cannot be the case that $\sigma = \sigma' 
[y^{E_1}. v_1 ,z^{E_2}. v_2]$ nor  $\sigma = \sigma' \exfalso_{P}$, otherwise we could apply an activation
or cross reduction.
Therefore $\sigma$ is case-free and does not contain $\exfalso_{P}$. Moreover,   $x^T$ cannot be a
free variable of $t$, then $T = A_i$, for some $1 \leq i \leq n$, and
$Q$ is a subformula of $A_i$, contradicting the assumptions. 

Suppose
therefore that $x^T$ is a bound intuitionistic variable of $t$, such
that $t$ has a subterm $\lambda x^T s: T \IMPL Y$ or, without loss of
generality, $s [x^{T}. v_1 ,z^{E}. v_2]$, with $s:T\VEL Y$ for some
formula $Y$. By induction hypothesis $T\IMPL Y$ and $T \VEL Y$ are subformulas of
$A_{1}, \ldots, A_{n}, A$ or proper subformulas of $D_{1}, \ldots,
D_{m}, G \IMPL H$. But $T\IMPL Y$ and $T \VEL Y$ contain $Q$ as a
proper subformula and $c^{F_1 \IMPL F_2} \, w$ violates maximally the
subformula due to $Q$. Therefore $T\IMPL Y$ and $T \VEL Y$ are neither
subformulas of $A_1 , \dots, A_n , A$ nor proper subformulas of $D_1 ,
\dots, D_m$ and thus must be proper subformulas of $G\IMPL H$. Since $c^{F_1 \IMPL F_2} \, w$ violates the subformula
property maximally due to $Q$, $T\IMPL Y$ and $T \VEL Y$ must be at
most as complex as $Q$, which is a contradiction.

Suppose now that $x^{T}$ is a bound channel variable, thus $w_{k}= a^T\, r\,\sigma$, where $a^{T}$ is a bound
communication variable of $t$, with $T = T_1 \IMPL T_2$. Since $c^{F_1
\IMPL F_2} \, w$ is rightmost, $a \neq c$.  Moreover, $Q$ is a
subformula of a prime factor of $T_2$, whereas $a^{T_{1}\IMPL T_{2}}$
occurs in $w$, which is impossible by choice of $c$. This
contradicts the assumption that the term is normal and ends the proof.
\end{itemize}
}
\end{proof}

\end{document}